\documentclass[12pt,a4paper]{amsart}

\usepackage[utf8]{inputenc}
\usepackage[T1]{fontenc}
\usepackage[french,english]{babel}
\usepackage{amsmath, amssymb, amsthm,mathtools}
\usepackage{mathrsfs}
\usepackage{hyperref}
\usepackage{geometry}
\usepackage{enumitem}
\usepackage{graphicx}
\usepackage{cite}
\usepackage{enumerate}
\usepackage{tikz}
\usetikzlibrary{calc,arrows.meta,positioning}
\usepackage{xcolor}

\newtheorem{theorem}{Theorem}[section]
\newtheorem{lemma}[theorem]{Lemma}
\newtheorem{proposition}[theorem]{Proposition}

\theoremstyle{definition}
\newtheorem{definition}[theorem]{Definition}
\newtheorem{example}[theorem]{Example}

\theoremstyle{remark}
\newtheorem{remark}[theorem]{Remark}
\newcommand{\R}{\mathbb R}
\newcommand{\C}{\mathbb C}
\newcommand{\Z}{\mathbb Z}
\newcommand{\CP}{\mathbb{CP}}
\newcommand{\tr}{\mathrm{tr}}
\newcommand{\Res}{\mathrm{Res}}

\newcommand{\Mcal}{\mathcal M}
\newcommand{\Pcal}{\mathcal P}
\newcommand{\Dcal}{\mathcal D}
\newcommand{\SL}{\mathrm{SL}}

\usepackage{parskip} 
\title[Parabolic Higgs Bundles and Hyperpolygon Spaces]{Semiclassical Limits of Strongly Parabolic Higgs Bundles and Hyperpolygon Spaces}
\author[L. Heller]{Lynn Heller}
\address{Beijing Institute of Mathematical Sciences and Applications (BIMSA), Beijing, China}
\email{lynn@bimsa.cn}
\author[S. Heller]{Sebastian Heller}
\address{Beijing Institute of Mathematical Sciences and Applications (BIMSA), Beijing, China}
\email{sheller@bimsa.cn}
\author[C. Meneses]{Claudio Meneses}
\address{Mathematisches Seminar, Christian-Albrechts-Universit\"at zu Kiel, Germany}
\email{meneses@math.uni-kiel.de}

\date{\today}

\begin{document}

\noindent
\begin{abstract}
We investigate the Hitchin hyperk\"ahler metric on the moduli space of strongly parabolic
$\mathfrak{sl}(2,\C)$-Higgs bundles on the $n$-punctured Riemann sphere and its
degeneration obtained by scaling the parabolic weights $t\alpha$ as $t\to0$.
Using the parabolic Deligne--Hitchin moduli space, we show that twistor lines of
hyperpolygon spaces arise as limiting initial data for twistor lines at small weights,
and we construct the corresponding real-analytic families of $\lambda$-connections.
On suitably shrinking regions of the moduli space, the rescaled Hitchin metric converges,
in the semiclassical limit, to the hyperk\"ahler metric on the hyperpolygon space
$\mathcal X_\alpha$, which thus serves as the natural finite-dimensional model for the
degeneration of the infinite-dimensional  hyperk\"ahler reduction.
Moreover, higher-order corrections of the Hitchin metric in this semiclassical regime
can be expressed explicitly in terms of iterated integrals of logarithmic differentials
on the punctured sphere.
\end{abstract}

\maketitle
\setcounter{tocdepth}{1}
\tableofcontents
\noindent

\section{Introduction}
Moduli spaces of Higgs bundles on Riemann surfaces represent a fundamental class of
examples of hyperk\"ahler manifolds which arise from gauge theory via infinite-dimensional
hyperk\"ahler reduction.
In his work \cite{Hitchin1987}, Hitchin showed that the moduli space of stable
Higgs bundles over a compact Riemann surface carries a natural hyperk\"ahler metric, and
as a real manifold, it is diffeomorphic to the moduli space of flat irreducible connections via
solutions of the self-duality equations.
Simpson~\cite{Sim2} later extended this correspondence to noncompact Riemann surfaces using
parabolic bundles and proving a non-abelian Hodge correspondence between stable
parabolic Higgs bundles and flat connections with prescribed local monodromy. These moduli spaces have been studied extensively from algebro-geometric, 
analytic, differential-geometric and topological perspectives.

In this work we will focus on \emph{strongly parabolic} 
$\mathfrak{sl}(2,\C)$-Higgs bundles  over the Riemann sphere. 
Thus we fix a finite subset $\{p_{1},\dots,p_{n}\}\subset\CP^1$ and consider
holomorphic bundles $E\rightarrow\CP^1$ of degree 0 and rank 2 enhanced 
with complex lines $\ell_j\subset E\vert_{p_j}$ and parabolic weights $0<\alpha_j<\tfrac12$. 
The parabolic Higgs fields $\Phi$ are constrained to have at most simple poles 
at each ${p_j}$, whose residues are nilpotent and annihilate each 
 line $\ell_j$. For generic weights $\alpha=(\alpha_1,\dots,\alpha_n)$, the resulting moduli space 
$\Mcal_{\mathrm{Higgs}}(\alpha)$ is a smooth complex symplectic manifold of dimension 
$2n-6$. The birational geometry resulting from the dependence on parabolic weights was 
described by Thaddeus \cite{Thaddeus}. On the analytic side, Konno \cite{Konno} constructed a complete 
hyperk\"ahler metric on $\Mcal_{\mathrm{Higgs}}(\alpha)$ which is a natural generalization 
of the Hitchin metric in the case of compact Riemann surfaces without marked points. 
As the construction proceeds through an infinite-dimensional hyperk\"ahler quotient, 
the metric is inherently non-explicit. In particular, even though the real-analytic dependence of the metric on 
parabolic weights follows from the work of Kim--Wilkin \cite{KiWi}, 
it is by no means accessible in closed form.

Simpson's parabolic non-abelian Hodge correspondence determines a 
real-analytic diffeomorphism between $\Mcal_{\mathrm{Higgs}}(\alpha)$ and the moduli space 
$\Mcal_{\mathrm{dR}}(\alpha)$ of irreducible logarithmic $\SL(2,\C)$-connections on the punctured surface
$\Sigma_{0} =\CP^1\setminus\{p_1,\dots,p_n\}$ 
with fixed 
eigenvalues $\pm\alpha_j$ of its residues. A natural framework to encode 
this correspondence is provided by the
\emph{parabolic Deligne--Hitchin moduli space}
$\Mcal_{\mathrm{DH}}^{\mathrm{par}}$,
which parametrizes parabolic $\lambda$-connections and interpolates
between parabolic Higgs bundles and logarithmic connections.
From the twistorial point of view, solutions of Hitchin's self-duality equations are
determined as distinguished \emph{real holomorphic sections} of
$\Mcal_{\mathrm{DH}}^{\mathrm{par}}\to\CP^1$, called \emph{twistor lines}. The full
hyperk\"ahler metric can be recovered from the geometry of these sections via the {\em twisted
holomorphic symplectic form} \cite{Sim3,HKLR}. This perspective played a central role in the 
analysis of the large-scale asymptotic geometry of Hitchin moduli spaces in the work of
Gaiotto--Moore--Neitzke \cite{GMN}.

We study an analogous \emph{semiclassical limit} of
parabolic Higgs bundle moduli spaces obtained by scaling down the parabolic weights.
We fix a generic weight vector $\alpha=(\alpha_1,\dots,\alpha_n)$
and consider the one-parameter family
\[
\Mcal_\alpha(t)=\Mcal_{\mathrm{Higgs}}(t\alpha), \qquad t\ge 0.
\]
For sufficiently small $t>0$, the moduli spaces $\Mcal_\alpha(t)$ are smooth and
biholomorphic to each other. However, as $t\to0$ the parabolic structure degenerates, and the nilpotent cone of the
Higgs bundle moduli space --- the central fiber of the Hitchin fibration, together
with a shrinking neighbourhood thereof --- collapses in the limit. 
In this regime of parabolic weights, Godinho--Mandini~\cite{GoMa} established an algebro-geometric correspondence between a Zariski open subset of $\Mcal_\alpha(t)$  and the hyperpolygon space 
$\mathcal X_\alpha$ introduced by Konno~\cite{Konno2}. This correspondence was later shown in~\cite{BiFloGoMa} to extend at the level of holomorphic symplectic structures.
It is thus natural to ask whether the 
corresponding hyperk\"ahler metrics are also related. 
Addressing this question is the main theme of this work.

The hyperpolygon spaces are defined as hyperk\"ahler 
quotients of a finite-dimensional quaternionic vector space, and in this sense they are natural generalizations of the Eguchi--Hanson space.  
The conceptual analogy between the 
Eguchi--Hanson metric and the Hitchin metrics was already observed in \cite{Hitchin91}.
In the case under consideration,
the analogy becomes explicit in Konno's constructions
 \cite{Konno, Konno2}, as
 the geometry governing the residue data of parabolic Higgs fields at each marked point is that of 
 an Eguchi--Hanson space. 
From a geometric perspective, our results show that the hyperk\"ahler quotient description 
of hyperpolygon spaces can be viewed as a semiclassical limit of the infinite-dimensional 
hyperk\"ahler reduction defining the Hitchin moduli space. We show that in the $t=0$ limit, 
the twisted hyperk\"ahler moment maps degenerate to  those of the hyperpolygon space, and the 
hyperpolygon spaces appear as a natural resolution of the singularity at the origin of the
$\Mcal_\alpha$ for $\alpha=0.$
Our motivation is the work \cite{HellerHellerTraizet2025} by the first two authors in 
collaboration with Traizet, 
 where, in the case $n=4$, all four parabolic weights are taken to be equal and simultaneously scaled to zero, and an orbifold quotient of the Eguchi--Hanson space is recovered as a geometric model for the degeneration of the Hitchin moduli space of the four-punctured sphere.

At a technical level, our strategy is a natural extension of that 
of \cite{HellerHellerTraizet2025}; we show 
that the twistorial data 
associated to a stable hyperpolygon can be used to construct 
twistor lines for the
Hitchin moduli space with parabolic weights $t\alpha_j$ on the 
corresponding stable strongly parabolic Higgs bundle for sufficiently small $t>0$.
This is achieved by formulating the problem of constructing real holomorphic sections as a monodromy problem for loop-valued connections and solving it via an implicit function theorem on suitable Banach spaces of loop algebras. In this formulation, the initial conditions for the implicit function theorem are naturally identified with the twistor lines of the hyperpolygon spaces.
We then show that the resulting real holomorphic sections are in fact twistor
lines, corresponding to genuine solutions of Hitchin's self-duality equations with
parabolic weights $t\alpha$. Finally, after rescaling  by $t^{-1}$, 
the Hitchin metric converges on subsets of bounded energy of the form $\mathcal E_t \le t\,C$
 to the hyperpolygon metric on $\mathcal X_\alpha$ as $t\to0$. 
As in  the setup of Gaiotto--Moore--Neitzke \cite{GMN} for large parabolic Higgs fields, the correction terms of the asymptotic expansion of the 
hyperk\"ahler metric in the degeneration parameter $t$ can be computed explicitly. 
The Deligne--Hitchin twistorial interpretation leads to their description in terms of 
iterated integrals  (polylogarithms) on the $n$-punctured sphere.
Furthermore,
from a Hamiltonian point of view, 
the rescaled Hitchin energy generating the natural $S^1$--action on $\Mcal_\alpha(t)$ also converges 
to the hyperpolygon Hamiltonian generating the corresponding circle action on $\mathcal X_\alpha$. 
In this sense  the hyperpolygon spaces are semiclassical limits 
of Hitchin moduli spaces.

Our results fit into a broader program aimed at understanding degenerations and
asymptotic regimes of non-abelian Hodge moduli spaces. Related phenomena appear 
in the study of wild character varieties and irregular
connections by Boalch \cite{Boalch08, Boalch12, Boalch, Boalch18},
 as well as in the 
twistorial analysis of hyperk\"ahler metric
degenerations developed by Gaiotto--Moore--Neitzke \cite{GMN}.
Studies of the asymptotic geometry of the Hitchin metric by Mazzeo--Swoboda--Wei\ss--Witt~\cite{MSWW} and Fredrickson~\cite{Fred} focus on degeneration regimes in which parabolic Higgs fields are scaled to infinity. A key distinction in the present work lies in the nature of the degeneration: here it concentrates near the core of the moduli space and is driven instead by simultaneously scaling the parabolic weights to zero.
In the simplest nontrivial case $n=4,$ Heller--Heller--Traizet~\cite{HellerHellerTraizet2025} study precisely this degeneration by scaling four equal weights to zero. In this setting, the correspondence between Hitchin metrics and ALG gravitational instantons established by Fredrickson--Mazzeo--Swoboda--Wei\ss~\cite{FMSW} may be juxtaposed with the Kummer-type construction of Biquard--Minerbe~\cite{BiquardMinerbe}.
This comparison, together with earlier work \cite{HHS} of the first two authors, suggests that our methods can be applied to establish an analogous degeneration governed by Eguchi--Hanson spaces, arising from deformations of the parabolic weights in a regime near the apex $\alpha=(\tfrac{1}{4},\tfrac{1}{4},\tfrac{1}{4},\tfrac{1}{4}).$
Finally, the extensions of the Godinho--Mandini correspondence
 to higher rank by Fisher--Rayan~\cite{FR16},  
  as well to higher genus by Rayan--Schaposnik~\cite{RS21}, provide natural frameworks in which the results of the present work may be generalized by replacing hyperpolygon spaces by Nakajima quiver varieties \cite{Nak}.

We would like to note that, after sharing a draft of our work,  Fredrickson and Yae \cite{FredYae26} informed us that they have independently obtained the same result in the case $n = 4$ using entirely different analytical methods.

\noindent
\textbf{Organization of the paper.}
In Section~\ref{sec:twis} we review the necessary background on strongly parabolic Higgs bundles, the
parabolic non-abelian Hodge correspondence, and the parabolic Deligne--Hitchin moduli
space.
In Section~\ref{sec:hp} we recall the construction of hyperpolygon spaces as hyperk\"ahler quotients
and their identification with moduli spaces of parabolic Higgs bundles for their respective $I$-complex structures.
Section~\ref{sec:construction} contains the core analytic results: the construction of real holomorphic
sections via a monodromy problem and the proof that these sections are twistor lines.
Technical lemmas are moved into the Appendix~\ref{appendix}.
Finally, in Section~\ref{sec:scl} we study the semiclassical limit of the Hitchin metric and prove
its convergence, after rescaling, to the hyperk\"ahler metric on the hyperpolygon space.

\section{Twistor theory for strongly parabolic Higgs bundles}\label{sec:twis}
The purpose of this section is to recall the twistorial description of the hyperk\"ahler
structure on moduli spaces of strongly parabolic Higgs bundles.
The Deligne-Hitchin moduli space 
encodes the complex-analytic facets of the non-abelian Hodge correspondence,
and
provides a unified framework in which the full hyperk\"ahler
geometry can be studied from holomorphic data.
In particular, solutions to Hitchin's self-duality equations correspond to distinguished
real holomorphic sections of the Deligne-Hitchin moduli space, and the hyperk\"ahler metric
is determined  via the twisted-holomorphic symplectic form.

\subsection{Strongly Parabolic Higgs Bundles}\label{sec:parahiggs}

Let $\Sigma$ be a compact Riemann surface and let ${\bf D} = p_1 + \cdots + p_n$ a divisor of
pairwise distinct points on $\Sigma$. 
A parabolic structure $\Pcal$  on a holomorphic rank-$2$ vector bundle 
$E\to \Sigma$ with trivial determinant is a choice at each puncture $p_j$ of a 
complex 
({\em quasi-parabolic}) line $\ell_j \subset E|_{{p_j}}$ together with a \emph{parabolic weight}
$0 < \alpha_j < \tfrac{1}{2}$  assigned to $\ell_j$.  
The parabolic degree $\mathrm{pardeg}(L)$  of a holomorphic line subbundle $L$ of $E$ with respect to $\Pcal$
is defined to be \[\mathrm{pardeg}(L)=\mathrm{deg}(L)+\sum_j\nu_j\]
where $\nu_j=\alpha_j$ if $L_{p_j}=\ell_j$ and $\nu_j=-\alpha_j$ otherwise.
A holomorphic bundle $E$ equipped with parabolic structure $\Pcal = (p_j,\ell_j,\alpha_j)_{j=1,..,n}$ is called {\em stable} if for every holomorphic line subbundle $L \subset E$, one has 
\begin{equation}\label{eq:slopestab}\mathrm{pardeg}(L) < \mathrm{pardeg}(E):=0.\end{equation}
By  Mehta-Seshadri  \cite{MeS} (see \cite{AW98, BY99} for the conventions used here), stable parabolic bundles are in one-to-one correspondence to conjugacy classes of irreducible unitary representations of $\pi_1(\Sigma\setminus {\bf D})$ whose holonomy around each puncture $p_j$ has eigenvalues $e^{\pm 2\pi i \alpha_j}$. Equivalently, for each stable parabolic structure $\Pcal$ there is a unique unitary \emph{compatible logarithmic connection}, i.e. a logarithmic $\mathfrak{sl}(2,\C)$-connection $\nabla$ on $E$ with unitary monodromy and with simple poles along ${\bf D}$ which is {\em compatible} with $\Pcal$ in the sense that 
 its  residues $\Res_{p_j}\nabla$ 
have $\ell_j$ as eigenlines with eigenvalues $\alpha_j$.

A strongly parabolic  $\mathfrak{sl}(2,\C)$-{\em Higgs bundle} on $(\Sigma,{\bf D})$ consists of a pair $(\mathcal P,\Phi)$, where $\mathcal P$ is a parabolic structure on a holomorphic vector bundle $E$ as above, and
\[\Phi \in H^0\!\big(\Sigma,\, K_{\Sigma}\otimes \mathfrak{sl}(2,\C)\otimes \mathcal O({\bf D})\big)
\]
is a meromorphic $\mathfrak{sl}(2,\mathbb C)$-valued 1-form defining a \emph{strongly parabolic Higgs field} with respect to $\mathcal P.$ This means that, at each point $p_j\in{\bf D},$ the residue satisfies
\[(\operatorname{Res}_{p_j}\Phi)\,\ell_j = 0.
\]
Note that two logarithmic connections are compatible with the same parabolic structure $\Pcal$ if and only if their difference is a strongly parabolic Higgs field compatible with $\mathcal P$.

A parabolic Higgs bundle $(\Pcal,\Phi)$ as above  is \emph{stable}  if it satisfies the inequality \eqref{eq:slopestab} for every $\Phi$-invariant  subbundle.
We call $(\Pcal,\Phi)$ \emph{polystable} if it is either stable or a direct sum of two parabolic line bundles of parabolic degree zero with $\Phi=0$.

For a parabolic weight vector \(\alpha=(\alpha_1,\dots,\alpha_n)\) consider the moduli space 
\[\Mcal_{\mathrm{Higgs}}=\Mcal_{\mathrm{Higgs}}(\alpha)\] of polystable strongly parabolic Higgs bundles 
on \((\Sigma,{\bf D})\). An important feature of the parabolic setting is that the moduli problem depends 
on the choice of the weight vector \(\alpha\). 
The space of admissible weights can be identified with a convex bounded polytope, 
which is subdivided by a finite collection of hyperplanes, called 
\emph{semi-stability walls}. By the definition of parabolic stability, semi-stability walls 
correspond to hyperplanes of the form $\mathrm{pardeg}(L)=0$ intersecting $(0,\frac{1}{2})^{n}$
nontrivially for some holomorphic line subbundle $L\subset E$ (see \cite{BY96} for details). They can be expressed in the form
\begin{equation}\label{eq:par-wall}
\sum_{j\in S}\alpha_j-\sum_{j\notin S}\alpha_j = -\deg(L),
\end{equation}
where $S\subset \{1,\dots,n\}$ is the index subset of quasi-parabolic lines interpolated by $L$.

The connected components of the complement of the union of these walls are convex 
regions known as \emph{open chambers}. 
For weights lying in the interior of an open chamber, every polystable strongly 
parabolic Higgs bundle is in fact stable. 
We refer to \cite{BY96, Thaddeus, MenesesSIGMA} for a detailed discussion of this 
wall-and-chamber structure.

 {\bf Convention:} 
In the following we will only consider parabolic weights $(\alpha_1,\dots,\alpha_n)$ that are \emph{generic}, i.e. not contained in a semi-stability wall in the weight polytope, and \emph{small} in the sense that the following inequality is satisfied
 \begin{equation}\label{eq:polygon-wall}
 \sum_{j=1}^{n}\alpha_{j}<1.
 \end{equation}
The genericity constraint implies that every polystable Higgs pair is in fact stable; in particular, both the de~Rham moduli space and the parabolic Higgs bundle moduli space are smooth complex manifolds. The small weight constraint is justified in Remark \ref{rem:small-weight} and ensures that no semi-stability wall is crossed by scaling down a generic choice of $\alpha$ to the vertex $(0,\dots,0)$. It is important to remark that the only semi-stability walls intersecting the small-weight region nontrivially are those for which the right-hand side in \eqref{eq:par-wall} is equal to 0, and these are precisely the walls that remain invariant under parabolic-weight scaling.

\subsection{Simpson's Non-Abelian Hodge Correspondence}
Hitchin \cite{Hitchin1987} showed that for a compact Riemann surface (without punctures) the moduli space of stable Higgs bundles carries a natural hyperk\"ahler metric, and is in particular diffeomorphic, via solutions of Hitchin's self-duality equations, to the moduli space of flat $\mathrm{SL}(2,\C)$-connections.
See \cite{Sim1} for the generalization to the $\mathrm{SL}(n,\C)$-case.
Simpson \cite{Sim2} later extended this \emph{non-abelian Hodge correspondence} to non-compact Riemann surfaces by introducing parabolic Higgs bundles.

\subsubsection*{Tame harmonic metrics}
The parabolic non-abelian Hodge correspondence between the moduli space of strongly parabolic Higgs bundles and the moduli space of logarithmic connections is realized via \emph{tame harmonic metrics}, or equivalently via solutions of Hitchin's self-duality equations satisfying prescribed growth conditions at the punctures $p_j$, determined by the parabolic weights.

More precisely, given a strongly parabolic Higgs pair $(\Pcal,\Phi)$, there exists a unique Hermitian metric $h$ on $E|_{\Sigma\setminus{\bf D}}$ of determinant 1 such that the associated Chern connection $D_h$ satisfies the self-duality equations
\[ F_{D_h} + [\Phi,\Phi^{*_h} ]= 0, \qquad \bar\partial^{D_h}\Phi = 0,
\]
on $\Sigma\setminus{\bf D}$,
and is {\em tame}
 around each $p_j\in{\bf D}$, i.e. in a local coordinate $z$ centered at $p_j$, it has the model growth $(z\bar z)^{\alpha_j}$; see  \cite[Section 2]{Sim2} (or Section~\ref{sec:constwistor} below) for details.

The metric $h$ determines an equivariant harmonic map from the universal cover of $\Sigma\setminus{\bf D}$ into hyperbolic $3$-space.
Here the hyperbolic $3$-space is identified with the space of positive-definite Hermitian inner products of determinant~$1$, and the equivariance is with respect to the monodromy of the flat $\mathrm{SL}(2,\C)$-connection
\begin{equation}\label{eq:flatconnection} \nabla = D_h + \Phi + \Phi^{*_h},\end{equation}
obtained by adding the Higgs field and its $h$-adjoint $\Phi^{*_h}$ to the Chern connection.
Although $\Phi^{*_h}$ is in general singular at the points $p_j$, there is a natural extension 
of the underlying holomorphic bundle induced by $\bar\partial^\nabla$ across ${\bf D}$ which turns $\nabla$ into a logarithmic connection with residues $\Res_{p_j}\nabla$ conjugate to $\mathrm{diag}(\alpha_j,-\alpha_j)$.
In particular, 
the local monodromies lie in the prescribed conjugacy classes. This completes one direction of the correspondence.

Conversely, starting from an irreducible logarithmic connection on $\Sigma\setminus{\bf D}$ with the prescribed residue data, one recovers a parabolic Higgs bundle by solving the harmonic metric equations; see Donaldson's work \cite{Donald} in the compact case.
We refer to Simpson's original paper \cite{Sim2} for the full analytical details in the non-compact case. Simpson's result 
for the case of strongly parabolic $\mathfrak{sl}(2,\C)$-Higgs bundles
can be stated as follows.

\begin{theorem}\label{the:NAH}
(Parabolic non-abelian Hodge correspondence \cite{Sim2}).
Let $\Mcal_{\mathrm{Higgs}}$ be the moduli space of polystable strongly parabolic
$\mathfrak{sl}(2,\C)$-Higgs bundles on $(\Sigma,{\bf D})$ with fixed parabolic weights
$0<\alpha_j<\tfrac12$, and let $\Mcal_{\mathrm{dR}}$ be the moduli space of irreducible flat
$\mathrm{SL}(2,\C)$-connections on $\Sigma\setminus{\bf D}$ with fixed conjugacy classes of
local monodromy, specified by the eigenvalues $e^{\pm 2\pi i \alpha_j}$ at each $p_j\in{\bf D}$.
Then there exists a real-analytic diffeomorphism
\[
\Mcal_{\mathrm{Higgs}} \;\cong\; \Mcal_{\mathrm{dR}},
\]
which maps a strongly parabolic Higgs bundle $(\Pcal,\Phi)$ to the equivalence class of a
flat connection $\nabla$ (cf.~\eqref{eq:flatconnection}) arising from a tame solution of
Hitchin's self-duality equations for that Higgs data.
\end{theorem}

This correspondence is one-to-one between stable (respectively polystable) parabolic Higgs bundles and irreducible (respectively totally reducible) flat connections.
Moreover, as first shown by Konno \cite{Konno}, it induces a hyperk\"ahler structure on the moduli space: in particular,
$\Mcal_{\mathrm{Higgs}}$ and $\Mcal_{\mathrm{dR}}$ are identified as \emph{real} manifolds
of real dimension $4n+12g-12$, where $n$ is the number of 
marked points and $g$ is the genus of $\Sigma.$ They carry distinct (indeed anti-commuting) complex structures,
denoted $I$ and $J$ respectively, both compatible with the underlying hyperk\"ahler metric.

\subsection{Parabolic Deligne-Hitchin Moduli Spaces}\label{sec:pdhms}
 The different complex structures of a hyperk\"ahler manifold can be encoded in its twistor space; see~\cite{HKLR}.
 The complex-analytic construction of the twistor space of the Hitchin moduli space is
  due to Deligne (in unpublished notes) and Simpson \cite{Sim3}. In the strongly parabolic setting, this construction was carried out by Simpson \cite{Sim21} and by Alfaya-G\'omez \cite{AlGo}. We outline the construction here in the rank-$2$ case; see also~\cite{HellerHellerTraizet2025}.

Fix a generic parabolic weight vector
$\alpha$.
For each complex number $\lambda \in \C$, consider the moduli space of parabolic $\lambda$-connections on $(\Sigma,{\bf D})$.
By a parabolic $\lambda$-connection we mean a triple \[\mathcal D=(E,\Pcal,D^\lambda)\] consisting of a holomorphic bundle $E$ with parabolic structure $\Pcal$ as above together with
\[D^\lambda=\begin{cases}\lambda \partial^\nabla \quad \text{if } \lambda\neq0\\\Phi \quad \text{if }\lambda=0\end{cases},\]
where $\nabla$ is a logarithmic connection compatible with the parabolic structure $\Pcal$, and $\Phi$ is a strongly parabolic Higgs field.

Note that for a parabolic \(\lambda\)-connection, the residue 
\(\operatorname{Res}_{p_j} D^\lambda\) has the quasi-parabolic line 
\(\ell_j\) as an eigenline with eigenvalue \(\lambda\alpha_j\) at each 
\(p_j\in{\bf D}\). Concretely, in a local holomorphic 
trivialization near \(p_j\), the operator \(D^\lambda\) can be written as
\[
D^\lambda
= \lambda\,\partial
+ \begin{pmatrix}
\lambda\alpha_j & * \\
0 & -\lambda\alpha_j
\end{pmatrix}
\frac{dz}{z}
+ \text{(holomorphic $1$-form)},
\]
so that \(\ell_j\) is indeed an eigenline of the residue with eigenvalue 
\(\lambda\alpha_j\).

In particular, for \(\lambda=0\) the data \((E,\mathcal P,D^0)\) coincide 
with a strongly parabolic Higgs bundle, while for \(\lambda=1\) the 
operator \(D^1\) is an ordinary logarithmic connection compatible with the 
parabolic structure \(\mathcal P\). More generally, for \(\lambda\neq 0\), 
the operator \(\tfrac{1}{\lambda}D^\lambda\) is the \((1,0)\)-part of a 
logarithmic connection compatible with \(\mathcal P\). Conversely, every 
logarithmic connection \(\nabla\) compatible with \(\mathcal P\) gives 
rise, by scaling, to a parabolic \(\lambda\)-connection 
\((E,\mathcal P,\lambda\,\partial^\nabla)\).

 \subsubsection*{The Hodge moduli space}
Fix a generic parabolic weight vector
$\alpha$.
Let $\Mcal_{\mathrm{Hod}}$ denote the moduli space of 
pairs $(\lambda,\mathcal D)$, where $\lambda\in\C$ is arbitrary and  $\mathcal D$ is a polystable parabolic
$\lambda$-connections on $(\Sigma,{\bf D})$.
It is a smooth complex analytic variety equipped with a holomorphic projection
\[
\pi_{\mathrm{Hod}} \colon \Mcal_{\mathrm{Hod}} \to \C,
\]
sending  $(\lambda,\mathcal D)$ to the parameter
$\lambda$.
By construction, for each $\lambda\neq0$, the fiber 
is analytically isomorphic to the de~Rham moduli space $\Mcal_{\mathrm{dR}}$ of 
logarithmic connections with prescribed residues
\begin{equation}\label{eq:lambdaderham}\pi_{\mathrm{Hod}}^{-1}(\lambda)\cong\Mcal_{\mathrm{dR}}.\end{equation}
The  fiber $\pi_{\mathrm{Hod}}^{-1}(0)$ is naturally identified with the moduli
space $\Mcal_{\mathrm{Higgs}}$ of polystable strongly parabolic Higgs bundles.

\subsubsection*{The Deligne--Hitchin moduli space}
Let $\bar{\Sigma}$ denote the Riemann surface with complex structure conjugate to that
of $\Sigma$, and let $\bar{{\bf D}}$ denote the  divisor on $\bar \Sigma$ corresponding to
${\bf D}$ on $\Sigma$. Note that the deRham moduli spaces of logarithmic connections on $\Sigma$ and $\bar\Sigma$  with prescribed residue data at ${\bf D}$ respectively $\bar{\bf D}$
are naturally isomorphic. In fact, each of them is biholomorphic
(via the Deligne extension for the fixed choices of logarithms at the punctures)
to the
Betti moduli space
$\Mcal_{\mathrm{Betti}}$, consisting of irreducible representations of
$\pi_1(\Sigma\setminus{\bf D})$ with fixed local conjugacy classes up to overall conjugation.

Using \eqref{eq:lambdaderham} on $\Sigma$ and $\bar\Sigma$,  the moduli space of parabolic
$\lambda$-connections on $\Sigma$ is canonically identified with the corresponding
moduli space of parabolic
$\lambda^{-1}$-connections on $\bar{\Sigma}$ for each $\lambda\in\C^*$.
Using this identification, one glues $\Mcal_{\mathrm{Hod}}(\Sigma)$ and
$\Mcal_{\mathrm{Hod}}(\bar{\Sigma})$ along $\C^*$ to obtain the
\emph{parabolic Deligne--Hitchin moduli space}
\[
\pi \colon \Mcal_{\mathrm{DH}}^{\mathrm{par}} \to \CP^1.
\]
 
 By construction, the fibers of $\pi$ are given by
\[
\pi^{-1}(0) = \Mcal_{\mathrm{Higgs}}(\Sigma), \qquad
\pi^{-1}(\infty) = \Mcal_{\mathrm{Higgs}}(\bar{\Sigma}),
\]
while for every $\lambda\in\C^*$ one has by \eqref{eq:lambdaderham}
\[
\Mcal_{\mathrm{dR}} (\Sigma)\cong\pi^{-1}(\lambda) \cong \Mcal_{\mathrm{dR}} (\bar \Sigma).
\]

\subsubsection*{Real structure}
  The fibration $\Mcal_{\mathrm{DH}}^{\mathrm{par}}\to\CP^1$ 
comes equipped with a natural \emph{real structure} $\tau\colon \Mcal_{\mathrm{DH}}^{\mathrm{par}}\to\Mcal_{\mathrm{DH}}^{\mathrm{par}}$, i.e.\ an anti-holomorphic
involution covering the antipodal map on $\CP^1$,
\[
\lambda \longmapsto -\,\bar{\lambda}^{-1}.
\]
Concretely, $\tau$ sends a parabolic $\lambda$-connection $(E,\Pcal,D^\lambda)$ to the
parabolic $(-1/\bar{\lambda})$-connection on the same underlying $C^\infty$ bundle,
obtained by complex conjugation of the connection data and 
of the quasi-parabolic lines;
see \cite[Sections~2.3--2.4]{HellerHellerTraizet2025} for details.
In particular, $\tau$ exchanges the moduli spaces of parabolic Higgs bundles for $\Sigma$ and $\bar{\Sigma}$ 
as fibers over $\lambda=0$ and $\lambda=\infty$, respectively.

Let \(\mathcal D\) be a parabolic \(\lambda\)-connection on \(\Sigma\), and denote by 
\(\rho_{\mathcal D}:\pi_1(\Sigma\setminus{\bf D})\to \mathrm{SL}(2,\C)\) the associated monodromy 
representation. The \(\lambda\)-connection \(\tau(\mathcal D)\), defined on the conjugate 
Riemann surface \(\bar\Sigma\), determines a representation 
\(\rho_{\tau(\mathcal D)}:\pi_1(\bar\Sigma\setminus{\bf D})\to \mathrm{SL}(2,\C)\) which is 
complex conjugate to \(\rho_D\). Equivalently, \(\rho_{\tau(\mathcal D)}\) is naturally 
identified with the complex conjugate of the contragredient representation 
\(\overline{(\rho_{\mathcal D}^{T})^{-1}}\).

\subsection{Twistor Lines}
Consider a (tame) solution  $(E,\Pcal,\Phi,h)$ of Hitchin's self-duality equations, and define the
associated family of  connections
 \begin{equation}\label{eq:assocfam}
 \nabla^\lambda = D_h +\lambda^{-1} \Phi + \lambda\Phi^{*_h}.\end{equation}
 One finds that $ \nabla^\lambda$ is flat for all $\lambda\in\C^*$ due to the harmonicity of the metric $h$.
 Moreover, for each $\lambda\in\C^*$, $ \nabla^\lambda $ has a natural extension as 
 a logarithmic connection into the singular points $p_j\in{\bf D}$ by the results of Simpson \cite{Sim2}.

 Using the identification \eqref{eq:lambdaderham}
 one obtains a holomorphic family of parabolic
$\lambda$-connections
\begin{equation}\label{eq:twistor-lambda-line}
\lambda\in\C^* \mapsto (\lambda,\mathcal D^\lambda)
\end{equation}
parametrized by the spectral parameter $\lambda\in\C^*.$
Note that in general, the underlying parabolic structure of $\Dcal^\lambda$ varies with $\lambda$.
This family extends to $\lambda=0$ and coincides with the given strongly parabolic Higgs field at $\lambda=0$. At $\lambda=1$,  $\mathcal D^1$ is via the non-abelian Hodge correspondence the logarithmic connection in Theorem \ref{the:NAH}.
This construction was explained for compact surfaces in \cite{Sim3}, and 
 established for the strongly parabolic setup in full analytic
detail by Mochizuki 
\cite{Mochizuki2006,Mo}.
 
Given
a solution of Hitchin's
self-duality equations on a parabolic bundle,
the
family \eqref{eq:twistor-lambda-line} therefore defines a holomorphic section
\[
s\colon \CP^1 \longrightarrow \Mcal_{\mathrm{DH}}^{\mathrm{par}},\qquad
\lambda \longmapsto(\lambda,\Dcal^\lambda),
\]
which is called the \emph{twistor line} associated to the solution.
By construction, the section $s$ satisfies a reality condition with respect to the
real structure $\tau$ on $\Mcal_{\mathrm{DH}}^{\mathrm{par}}$, namely
\begin{equation}\label{eq:taurealse}
\tau\bigl(s(-\bar{\lambda}^{-1})\bigr)=s(\lambda)
\qquad\text{for all }\lambda\in\CP^1.
\end{equation}
A general holomorphic section $s$  of
$\Mcal_{\mathrm{DH}}^{\mathrm{par}}$ satisfying \eqref{eq:taurealse} is called {\em real holomorphic}.

\begin{remark}
Not every real holomorphic section of the Deligne--Hitchin moduli space arises
from a solution of Hitchin's equations.
First examples of real holomorphic sections which are not twistor lines were constructed
in \cite{HH}.
\end{remark}

 \subsection{Twisted holomorphic symplectic form for Fuchsian $\lambda$-connections}

The parabolic Deligne-Hitchin moduli space 
\(\mathcal M_{\mathrm{DH}}^{\mathrm{par}}\) carries a canonical fiberwise holomorphic
\(2\)-form twisted by \(\mathcal O(2)\). More precisely, there is a
natural holomorphic section
\[
\varpi \in H^0\!\bigl(\mathcal M_{\mathrm{DH}}^{\mathrm{par}},
\Lambda^2 T^*_{\mathrm{fib}} \otimes \mathcal O(2)\bigr),
\]
where \(T_{\mathrm{fib}}\) denotes the holomorphic tangent bundle along the fibers of
the projection \(\pi\colon \mathcal M_{\mathrm{DH}}^{\mathrm{par}}\to\CP^1\).
In the twistorial description of hyperk\"ahler manifolds \cite{HKLR}, \(\varpi\) plays
the role of the holomorphic symplectic form on the fibers.

After trivializing \(\mathcal O(2)\) over \(\lambda\in \C^*\),
the restriction of \(\varpi\) to the fiber \(\pi^{-1}(\lambda)\) can be written as
\begin{equation}\label{eq:varphilambda}
\varpi\big|_{\pi^{-1}(\lambda)}
=
\lambda^{-1}(\omega_J+i\omega_K)
-2\,\omega_I
-\lambda(\omega_J-i\omega_K),
\end{equation}
where \(\omega_I,\omega_J,\omega_K\) are the K\"ahler forms associated with the
hyperk\"ahler metric and the complex structures \(I,J,K\), respectively. This expression encodes the full hyperk\"ahler metric \cite{HKLR}.

In the present setting of Deligne-Hitchin moduli spaces, these forms admit concrete interpretations. Up to
normalization, \(\omega_J+i\omega_K\) coincides with the canonical holomorphic
symplectic form on the Higgs moduli space, while
\(\varpi|_{\pi^{-1}(1)}\) agrees with the Atiyah--Bott--Goldman holomorphic symplectic
form on the de~Rham moduli space. More generally, for each \(\lambda\in\C^*\),
\(\varpi|_{\pi^{-1}(\lambda)}\) is given by a Goldman-type holomorphic symplectic form
on the moduli space of parabolic \(\lambda\)-connections. 

In the following, we restrict
to \(\Sigma=\CP^1\) and to the Zariski-open subset of the de~Rham moduli space
consisting of logarithmic connections on the trivial bundle, i.e., classical
Fuchsian systems.
Then, \(\Omega_J=\varpi|_{\pi^{-1}(1)}\) can be written explicitly in
terms of residues;
see \cite{AleMale} or \cite[Section~3.4]{HellerHellerTraizet2025} for details.

 Let \({\bf D}=p_1+\dots+p_n\) with \(p_j\in\C\subset\CP^1\), and
consider a connection of the form
\[
\nabla = d + \sum_j A_j\,\frac{dz}{z-p_j},
\]
where \(A_j\in\mathfrak{sl}_2(\C)\) satisfy \(\det(A_j)=-\alpha_j^2\) and
\(\sum_j A_j=0\), ensuring regularity at infinity.
Tangent vectors \(X,Y\) to the moduli space at \(\nabla\) may be represented by
meromorphic \(\mathfrak{sl}_2(\C)\)-valued \(1\)-forms
\[
X=\sum_j X_j\,\frac{dz}{z-p_j}, \qquad
Y=\sum_j Y_j\,\frac{dz}{z-p_j},
\]
with \(\sum_j X_j=\sum_j Y_j=0\) and
\(\tr(A_j X_j)=\tr(A_j Y_j)=0\) for all \(j\), expressing infinitesimal preservation of
the residue conjugacy classes.
The
holomorphic symplectic $\Omega_J$ form arises via symplectic reduction for the conjugation action
of \(\mathrm{SL}(2,\C)\) with moment map \(\mu=\sum_j A_j\), and is given by the residue
formula
\begin{equation}\label{eq:Goldman-Kirilov}
\Omega_J(X,Y)
= \sum_j \frac{1}{8\,\tr(A_j^2)}\,
\tr\!\bigl(A_j[X_j,Y_j]\bigr).
\end{equation}
Each summand is the Kirillov--Kostant--Souriau form on the adjoint orbit of \(A_j\).

For arbitrary \(\lambda\in\C\), an entirely analogous residue formula describes the holomorphic symplectic form  \(\varpi|_{\pi^{-1}(\lambda)}\) on the
moduli space of parabolic \(\lambda\)-connections; compare with \cite[Section 2.7]{HellerHellerTraizet2025}. For \(\lambda=0\) and weights lying in 
the proper subpolytope for which stable parabolic bundles exist \cite{AW98, BY96}, this reduces (up to scale) to the
canonical Liouville symplectic form on the cotangent bundle of the moduli space of stable parabolic bundles \cite{Hitchin1987, Konno}.

\section{Hyperpolygon spaces}\label{sec:hp}
In this section we recall the definition and basic properties of hyperpolygon spaces,
viewed as finite-dimensional hyperk\"ahler quotients. These spaces provide explicit 
hyperk\"ahler models which are useful to study the geometry of strongly parabolic 
Higgs bundles with complex structure $I$  \cite{GoMa}.
Our main interest in the hyperpolygon spaces lies in their role as  models
for the degeneration of the parabolic Higgs bundle moduli as the parabolic weights
tend to zero.
In particular, their hyperk\"ahler quotient description and explicit twistor geometry
will later allow us to compare the limiting hyperk\"ahler metric with the rescaled
Hitchin metric arising from the parabolic Higgs bundle side.

\subsection{The Eguchi--Hanson space}

The holomorphic cotangent bundle $T^*\CP^1$ is naturally equipped with a holomorphic
symplectic form $\Omega$.
Moreover, for each $\alpha\in\R_{>0}$ there exists a K\"ahler form
$\omega=\omega_\alpha$ on $T^*\CP^1$ whose restriction to the zero section
$\CP^1\subset T^*\CP^1$ is the multiple $\alpha\,\omega_{\mathrm{FS}}$ of the
Fubini--Study K\"ahler form.
Together, $(\omega,\Omega)$ determine a hyperk\"ahler structure on $T^*\CP^1$.
The resulting metric is the well-known Eguchi--Hanson metric \cite{Calabi}.
It is most conveniently constructed via the hyperk\"ahler quotient construction; see
\cite{HKLR}. We briefly recall this construction and refer to \cite{Hitchin91, Konno2} for further
details.

Let $V=\C^2$ and consider the vector space $V\oplus V^*$.
We equip $V$ (and $V^*$) with the standard Hermitian inner product $(\cdot,\cdot)$.
Let
\[
{}^*\colon V\to V^*,\qquad v\mapsto v^*:=(\cdot,v),
\]
denote the complex anti-linear isomorphism induced by the Hermitian form, and use the
same symbol for its inverse, so that $(\cdot,w^*)=w$ for $w\in V^*$.

Then $V\oplus V^*$ carries a natural flat hyperk\"ahler structure
$
(g_0,I,J,K),
$
where $g_0=\mathrm{Re}(\cdot,\cdot)$ is the Euclidean metric on
$M:=V\oplus V^*\cong\R^8$, the complex structure $I$ is multiplication by $i$, and
\[
J(v,w)=(-w^*,v^*)\in V\oplus V^*.
\]

Writing $z_\ell=x_\ell+iy_\ell$, the K\"ahler form and holomorphic symplectic form
with respect to $I$ are
\[
\omega=\omega_I:=\mathrm{Im}(\cdot,\cdot)=-g_0(\cdot,I\cdot)
=\frac{i}{2}\sum_{\ell}dz_\ell\wedge d\bar z_\ell,
\]
and
\[
\Omega=\Omega_I:=\omega_J+i\omega_K
=dz_1\wedge dz_3+dz_2\wedge dz_4.
\]
Consider the $S^1$-action on $V\oplus V^*$ given by
\[
\varphi\cdot(v,w):=(\varphi v,\bar\varphi\, w).
\]
The associated Killing field is given by
\[
X_{(v,w)}=(iv,-iw).
\]

\begin{lemma}
The $S^1=\mathrm{U}(1)$-action is \emph{tri-holomorphic}, i.e.\ it preserves $I$, $J$, and $K$, and it is
\emph{tri-hamiltonian}, i.e.\ Hamiltonian with respect to $\omega_I$, $\omega_J$, and $\omega_K$.
Identifying $\mathfrak{u}(1)=i\R\cong\R$, the moment maps are given by
\[
\mu_I(v,w)=\tfrac12\bigl(|v|^2-|w|^2\bigr),
\]
and
\[
\mu_\C(v,w):=(\mu_J+i\mu_K)(v,w)=i\,w(v).
\]
\end{lemma}

\noindent

For $\alpha>0$, the flat hyperk\"ahler structure on $V\oplus V^*$ descends to a complete
hyperk\"ahler metric on the quotient
\[
\mu^{-1}(2\alpha,0)\big/S^1 \;\cong\; T^*\CP^1,
\]
where $\mu=(\mu_I,\mu_\C)$ is the full hyperk\"ahler moment map. 
This hyperk\"ahler quotient  is the \emph{Eguchi--Hanson space}, 
the simplest non-flat asymptotically locally Euclidean hyperk\"ahler manifold. The K\"ahler form with respect to 
the complex structure $I$ is the scaled Fubini-Study metric $\alpha\,\omega_{\mathrm{FS}}$ on the zero section $\CP^1$.

There is also a natural action of $\mathrm{SU}(2)$ on $V\oplus V^*$ given by
\[
A\cdot(v,w)=(Av,w\circ A^{-1}).
\]
For $\xi\in\mathfrak{su}(2)\cong\mathfrak{su}(2)^*$, the associated Killing field is
\[
X_\xi(v,w)=(\xi v,-w\xi).
\]

\begin{lemma}\label{lem:1legmoment}
The $\mathrm{SU}(2)$-action on $V\oplus V^*$ is tri-holomorphic and tri-hamiltonian.
The corresponding moment maps are
\[
\nu_I(v,w)=\tfrac{i}{2}\bigl(v\otimes v^*-w^*\otimes w\bigr)_0,
\]
and
\[
\nu_\C(v,w):=(\nu_J+i\nu_K)(v,w)=-(v\otimes w)_0,
\]
where $(\cdot)_0$ denotes the trace-free part of an element of $\mathfrak{gl}(2,\C)$.
\end{lemma}

\subsubsection{Twistor lines}\label{sec:hptwistor}
Let $(x,y)\in \C^2\setminus\{0\}$ and 
\[a=\frac{\sqrt{2\alpha+|x|^2+|y|^2}}{\sqrt{|x|^2+|y|^2}}>0.\]
Define
\begin{equation}\label{eq:vwxy}
v=a\begin{pmatrix}x\\y\end{pmatrix}\in V,
\qquad
w=(y,-x)\in V^*.
\end{equation}
Then
\[
\mu_I(v,w)=2\alpha,
\qquad
\mu_\C(v,w)=0.
\]

Define
\begin{equation}\label{def:mfA}
\begin{split}
\mathfrak A
:=&\lambda^{-1}\nu_\C(v,w)-2i\,\nu_I(v,w)-\lambda\,\nu_\C(w^*,v^*)\\
=&\lambda^{-1}a\begin{pmatrix}-xy& x^2\\y^2&xy \end{pmatrix}
+\frac{1 + a^2}{2}\begin{pmatrix}|x|^2-|y|^2& 2x\bar y\\ 2y\bar x& |y|^2-|x|^2 \end{pmatrix}
+\lambda a\begin{pmatrix} \bar x \bar y &\bar y^2\\-\bar x^2&-\bar x\bar y  \end{pmatrix}
\end{split}
\end{equation}
It satisfies
\begin{equation}\label{eq:SDforRes}
\det(\mathfrak A)=-\alpha^2,
\qquad
\mathfrak A^*=-\mathfrak A,
\end{equation}
where, for $A\colon\C^*\to\mathfrak{sl}(2,\C)$,
\begin{equation}\label{defstar}
A^*(\lambda):=-\overline{A(-\bar\lambda^{-1})}^{\,T}.
\end{equation}
Its quasi-parabolic line, i.e. the kernel of the $\lambda^{-1}$-part,  is the span of $v.$
\begin{remark}\label{rem:twisted}
Up to the factor $i$, the matrix $\mathfrak A$ coincides with the twisted-holomorphic
moment map for the complexified $\mathrm{SL}(2,\C)$-action with respect to the
twisted-holomorphic symplectic form
\[
\Omega_\lambda=\lambda^{-1}\Omega_I-2\omega_I-\lambda\,\bar\Omega_I.
\]
This is the standard description of the twistor space of a hyperk\"ahler quotient;
see \cite{HKLR}. For the  twistorial description of (co)adjoint orbits
and the corresponding hyperk\"ahler structures, see also \cite{Biquard98}.
\end{remark}

\subsubsection{Bottom Component}\label{rem:Scaling}
Consider $\mathfrak A$ in \eqref{def:mfA} when
scaling the vector $\tilde w= r(y,-  x)$ to zero. In the limit 
$r\to0$  we obtain
\[\lim_{r\to0}\mathfrak A= \frac{\alpha}{|x|^2+|y|^2}\,
\begin{pmatrix}|x|^2-|y|^2&2x\bar y\\2 y\bar x&|y|^2-|x|^2
 \end{pmatrix}
\in i\mathfrak{su}(2)\]
which is independent of $\lambda.$ Its eigenvalues are $\pm\alpha,$ and the eigenline with respect to the eigenvalue $\alpha$ is 
given by  the quasi-parabolic line, i.e. it is spanned by the limit 
\[\lim_{r\to0}v=\sqrt{2\alpha}
\begin{pmatrix}x\\y\end{pmatrix}\neq0.\]

\subsection{The hyperpolygon space}

We next recall the \emph{hyperpolygon spaces} introduced by Konno \cite{Konno2} as
hyperk\"ahler quotients of a product of Eguchi--Hanson spaces. Equivalently, 
they can be obtained as hyperk\"ahler quotients of flat quaternionic vector spaces.
Let $n\geq 4$ and consider
\[
M_n := \bigoplus_{j=1}^n \bigl(V\oplus V^*\bigr)
\;\cong\; \C^{4n},
\]
equipped with the product hyperk\"ahler structure
$
(g_0,I,J,K),
$
where each summand $V\oplus V^*$ carries the flat hyperk\"ahler structure described
above.
We write points of $M_n$ as
\[
(v,w)=\bigl((v_1,w_1),\dots,(v_n,w_n)\bigr),
\qquad v_j\in V,\; w_j\in V^*.
\]

Let
\[
\mathrm{G}:=\mathrm{SU}(2)\times (S^1)^n
\]
act on $M_n$ as follows.
The $\mathrm{SU}(2)$-factor acts diagonally via the action described in the previous
subsection,
\[
A\cdot (v_j,w_j)=(Av_j,w_j\circ A^{-1}),
\]
while the $j$-th $S^1$-factor acts on the $j$-th summand by
\[
\varphi_j\cdot(v_j,w_j)=(\varphi_j v_j,\bar\varphi_j w_j),
\]
and trivially on the remaining summands.
Altogether, this defines a tri-holomorphic action of $\mathrm{G}$ on $M_n$.

Identifying $\mathfrak{u}(1)^n\cong\R^n$ and using the notation of the previous subsection,
the hyperk\"ahler moment map for the $(S^1)^n$-action is
$
\mu=(\mu_I,\mu_\C),
$
with components
\[
\mu_I(v,w)=\bigl(\tfrac12(|v_1|^2-|w_1|^2),\dots,\tfrac12(|v_n|^2-|w_n|^2)\bigr),
\]
and
\[
\mu_\C(v,w)=\bigl(i\,w_1(v_1),\dots,i\,w_n(v_n)\bigr).
\]

Similarly, the $\mathrm{SU}(2)$-moment map is the sum of the individual contributions,
$
\nu=(\nu_I,\nu_\C),
$
where
\[
\nu_I(v,w)
=\tfrac{i}{2}\sum_{j=1}^n\bigl(v_j\otimes v_j^*-w_j^*\otimes w_j\bigr)_0,
\]
and
\[
\nu_\C(v,w)
=-\sum_{j=1}^n (v_j\otimes w_j)_0.
\]

\begin{lemma}
The $\mathrm{G}=\mathrm{SU}(2)\times (S^1)^n$-action on $M_n$ is tri-hamiltonian with moment map
\[
(\nu,\mu)\colon M_n\to
\mathfrak{su}(2)^*\oplus(\R^n\oplus\C^n).
\]
\end{lemma}

Fix a \emph{hyperpolygon weight} vector
\[
\alpha=(\alpha_1,\dots,\alpha_n)\in\R_{>0}^n.
\]
The \emph{hyperpolygon space} associated to $\alpha$ is defined as the
hyperk\"ahler quotient
\[
\mathcal X_\alpha
:=\nu^{-1}(0)\cap
\mu_I^{-1}(2\alpha)\cap
\mu_\C^{-1}(0)
\;\big/\; \mathrm{G}.
\]

Equivalently, since $(S^1)^n$ and $\mathrm{SU}(2)$ commute,
$\mathcal X_\alpha$ may be viewed as the hyperk\"ahler quotient of $M_n$ by
$(S^1)^n$ at level $(2\alpha,0)$, followed by reduction by
$\mathrm{SU}(2)$ at level $0$.
By construction, $\mathcal X_\alpha$ carries a natural hyperk\"ahler metric induced from the flat metric on $M_n$.
The identification of this hyperk\"ahler quotient with a complex-analytic quotient,
and hence the global description of $\mathcal X_\alpha$ as a smooth hyperk\"ahler
manifold, relies on a Kempf--Ness--type result \cite{HKLR, Kirwan, Konno2}, which we recall below.

\subsection{Complex orbits, stability, and the hyperk\"ahler metric}
We now  discuss the notion of stability and compare the description of 
hyperpolygon spaces as hyperk\"ahler quotients with their interpretation as complex-analytic moduli spaces.
Recall the complex moment maps
\[
\mu_\C\colon M_n\to\C^n,
\qquad
\nu_\C\colon M_n\to\mathfrak{sl}(2,\C),
\]
and consider the complex zero set
\[
\hat{\mathcal X}:=\nu_\C^{-1}(0)\cap\mu_\C^{-1}(0).
\]
For a point
\[
(v,w)=\bigl((v_1,w_1),\dots,(v_n,w_n)\bigr)\in\hat{\mathcal X},
\]
the defining equations are
\[
w_j(v_j)=0,
\qquad
\sum_{j=1}^n (v_j\otimes w_j)_0=0.
\]
These equations are preserved by the holomorphic action of
\[
\mathrm{G}^\C=\mathrm{SL}(2,\C)\times(\C^*)^n,
\]
and hence determine a complex orbit
\[
\mathcal O_{(v,w)}:=\mathrm{G}^\C\cdot(v,w)\subset\hat{\mathcal X}.
\]

Fix a hyperpolygon weight vector $\alpha=(\alpha_1,\dots,\alpha_n)\in\R_{>0}^n$.
As shown by Konno \cite{Konno2}, the choice of $\alpha$ determines a notion of
(semi)stability for the $\mathrm{G}^\C$-action on $\hat{\mathcal X}$. We follow the 
stability notion introduced by Konno as adapted by Godinho--Mandini \cite{GoMa} (see \cite{Klyachko} for its introduction on moduli of polygons). This definition is equivalent to Nakajima's original definition; see \cite{Nak}.

A hyperpolygon weight vector $\alpha$ is called \emph{generic} if for every  subset
$S\subset\{1,\dots,n\}$ one has
\[
\epsilon_{S}(\alpha):=\sum_{j\in S}\alpha_j-\sum_{j\notin S}\alpha_j\neq0.
\]
In the following, we restrict to generic hyperpolygon weights.

A subset $S\subset\{1,\dots,n\}$ is called \emph{$\alpha$-short} if $\epsilon_{S}(\alpha)<0$,
and $\alpha$-\emph{long} otherwise.
Given $(v,w)\in\hat{\mathcal X}$, a subset $S$ is called \emph{straight} at $(v,w)$ if
\[
\det(v_j,v_k)=0\qquad\text{for all }j,k\in S.
\]

\begin{definition}
A point $(v,w)=((v_1,w_1),\dots,(v_n,w_n))\in\hat{\mathcal X}$
is called \emph{$\alpha$-stable} if the following conditions hold:
\begin{itemize}
\item $v_j\neq0$ for all $j=1,\dots,n$;
\item if $S\subset\{1,\dots,n\}$ is straight at $(v,w)$ and $w_j=0$ for all
$j\notin S$, then $S$ is $\alpha$-short, i.e.
\[
\sum_{i\in S}\alpha_i \;<\; \sum_{i\notin S}\alpha_i.
\]
\end{itemize}
If at least one of the previous
inequalities is non-strict
then $(v,w)$ is called \emph{$\alpha$-semistable}.
This can only happen for non-generic $\alpha$.
\end{definition}
We denote by
\[
\hat{\mathcal X}_\alpha^{\mathrm{s}},
\qquad
\hat{\mathcal X}_\alpha^{\mathrm{ss}}
\]
the $\alpha$-stable and $\alpha$-semistable loci, respectively.
If $\alpha$ is generic, then
$\hat{\mathcal X}_\alpha^{\mathrm{ss}}=\hat{\mathcal X}_\alpha^{\mathrm{s}}$.
It is known (see \cite{GoMa} and references therein) that in this case
$\hat{\mathcal X}_\alpha^{\mathrm{s}}$ is smooth.

It remains to solve the real moment map equations 
$\nu_I=0$ and $\mu_I=2\alpha$ 
inside a given complex orbit \(\mathcal O_{(v,w)}\).
In this setting, Konno \cite{Konno2} proved the following Kempf--Ness--type result.

\begin{proposition}\label{prop:HK-KN}
Let $(v,w)\in \hat{\mathcal X}$ and let $\alpha$ be generic.
Then the following are equivalent:
\begin{itemize}
\item $(v,w)$ is $\alpha$-stable for the $\mathrm{G}^\C$-action;
\item the complex orbit $\mathcal O_{(v,w)}$ contains a representative satisfying
the real moment map equations, i.e.\ there exists $g\in \mathrm{G}^\C$ such that
\[
g\cdot(v,w)\in \nu_I^{-1}(0)\cap\mu_I^{-1}(2\alpha).
\]
\end{itemize}
Moreover, such a representative is unique up to the action of the compact group
\[
\mathrm{G}=\mathrm{SU}(2)\times(S^1)^n.
\]
\end{proposition}

In particular, for $\alpha$-stable points the intersection
\[
\mathcal O_{(v,w)}\cap\bigl(\nu_I^{-1}(0)\cap\mu_I^{-1}(2\alpha)\bigr)
\]
is a single $\mathrm{G}$-orbit.
Consequently, for generic $\alpha$ the hyperpolygon space admits the equivalent
descriptions
\[
\mathcal X_\alpha
=
\nu_I^{-1}(0)\cap\mu_I^{-1}(2\alpha)\cap\hat{\mathcal X}\big/\mathrm{G}
\;\cong\;
\hat{\mathcal X}_\alpha^{\mathrm{s}}\big/\mathrm{G}^\C,
\]
and the resulting quotient is a smooth hyperk\"ahler manifold.
Geometrically, this description realizes $\mathcal X_\alpha$ as a moduli space of
weighted polygons in $\R^3$ equipped with additional ``momentum'' variables,
which motivates the terminology \emph{hyperpolygon space}.

\begin{remark}\label{rem:left-right-convention}
We emphasize a difference of conventions between the hyperpolygon and Higgs bundle
hyperk\"ahler quotient constructions.
On the hyperpolygon side, the $\mathrm{SL}(2,\C)$-action on the data $(v_j,w_j)$ is
written as a left action, as is customary in the hyperk\"ahler quotient and quiver
variety literature.
On the Higgs bundle side, by contrast, the $\mathrm{SL}(2,\C)$-action arises from the
gauge group and is therefore naturally a right action.
These conventions are equivalent: passing from a left to a right action amounts to
composition with inversion in $\mathrm{SL}(2,\C)$.
All constructions and statements below are independent of this choice of convention.
\end{remark}

\subsubsection*{Hyperk\"ahler metric}

We now describe the hyperk\"ahler metric on the hyperpolygon space
$\mathcal X_\alpha$ from the twistorial point of view in the spirit of 
\cite{Biquard98}. The description below combines the hyperk\"ahler quotient
construction of hyperpolygon spaces \cite{Konno2} with the general fact that
hyperk\"ahler quotients inherit the full twistor family of holomorphic symplectic
forms \cite{HKLR}. This interpretation provides a convenient
reformulation of the standard hyperk\"ahler reduction picture which makes the
dependence on the twistor parameter explicit.
In addition, we use the identification of the holomorphic
Kirillov--Kostant--Souriau form on (co)adjoint orbits with the canonical
holomorphic symplectic form on (twisted) cotangent bundles.
For a single factor $V\oplus V^*$, this reduces to the explicit description of the
$\lambda$-dependent holomorphic moment map recalled in
Remark~\ref{rem:twisted}; see also \cite[Th\'eor\`eme~4]{Biquard98}.

Fix a generic hyperpolygon weight vector $\alpha=(\alpha_1,\dots,\alpha_n)\in\R_{>0}^n$ and
consider
\[
M_n:=\bigl(V \oplus V^*\bigr)^n\]
with its flat hyperk\"ahler structure $(g,I,J,K)$ and the tri-hamiltonian action of
\[
\mathrm{G}=\mathrm{SU}(2)\times (S^1)^n
\]
as described above.
Let $\omega_I$ denote the K\"ahler form and
$\Omega_I:=\omega_J+i\omega_K$ the $I$-holomorphic symplectic form on $M_n$.
For $\lambda\in\C^*$, consider the \emph{twisted-holomorphic symplectic form}
\[
\Omega_\lambda:=\lambda^{-1}\Omega_I-2\,\omega_I-\lambda\,\overline{\Omega}_I.
\]

For $(v_j,w_j)\in V\oplus V^*$, let
\[
\nu_I(v_j,w_j)=\frac{i}{2}\,(v_j\otimes v_j^* - w_j^*\otimes w_j)_0,\qquad
\nu_\C(v_j,w_j)=-(v_j\otimes w_j)_0
\]
be the $\mathrm{SU}(2)$ moment maps as in Lemma~\ref{lem:1legmoment}.  
\begin{remark}\label{rem:momentt0}
The formal similarity  of the moment maps for the hyperk\"ahler reduction of the hyperpolygon spaces, with the moment maps
of the Hitchin equations has been observed in various places, see for example \cite{BiquardMinerbe} and \cite{RS21}.
\end{remark}

Define the $\lambda$--dependent $\mathfrak{sl}(2, \mathbb C)$-matrix
\begin{equation}\label{eq:mafraA}
\mathfrak A(v_j,w_j):=\;\lambda^{-1}\nu_\C(v_j,w_j)\;-\;2i\,\nu_I(v_j,w_j)\;-\;\lambda\,\nu_\C(w_j^*,v_j^*).
\end{equation}

If $(v_j,w_j)$ satisfies the $(S^1)$-moment map equations
$\mu_I(v_j,w_j)=2\alpha_j$ and $\mu_\C(v_j,w_j)=0$, then
$\mathfrak A(v_j,w_j)$ lies in the semisimple adjoint orbit in
$\mathfrak{sl}(2,\C)$ with eigenvalues $\pm\alpha_j$.
In this case, the map
\[(v_j,w_j)\mapsto i\;\mathfrak A(v_j,w_j)\]
is the twisted-holomorphic moment map (see \cite[Section 3 (F)]{HKLR}) for the complexified
$\mathrm{SL}(2,\C)$-action with respect to $\Omega_\lambda$.

\begin{theorem}\label{thm:twistor-hyperpolygon-orbits}
Let $\alpha$ be generic and let $p=(v,w)\in\mathcal X_\alpha$ be a point in the
hyperpolygon space.
Then the associated twistor line $s_p$ is given by
\[(\,\mathfrak A(v_1,w_1),\dots,\mathfrak A(v_n,w_n)\,).\] 
Moreover, the twisted-holomorphic symplectic form
 on $\mathcal X_\alpha$ is obtained by
holomorphic symplectic reduction of $\Omega_\lambda$ with respect to the
$\lambda$-dependent moment map
\[
\mu_\lambda(v,w)\;=\; i\;\sum_{j=1}^n \mathfrak A(v_j,w_j).\]
\end{theorem}

Theorem~\ref{thm:twistor-hyperpolygon-orbits} together with 
\eqref{eq:varphilambda}  gives an expression for the twisted-holomorphic symplectic form 
for the hyperpolygon space
 in terms of the Kirillov--Kostant--Souriau form on adjoint orbits.
We have seen in \eqref{eq:Goldman-Kirilov}  that
 the twisted-holomorphic symplectic form 
 on  the moduli of Fuchsian $\lambda$-connections
 is also determined in terms of residues.
This 
provides the link between the twistor geometry of hyperpolygon spaces and that of
parabolic Higgs bundle moduli spaces as discussed in Section~\ref{sec:twis}.

\subsection{Hyperpolygons and strongly parabolic Higgs bundles}
We recall the Godinho--Mandini correspondence \cite{GoMa} between the holomorphic quotients
$\hat{\mathcal X}_\alpha^{\mathrm{s}}/\mathrm{G}^\C$ and the moduli space of stable strongly parabolic
Higgs bundles on the trivial holomorphic bundle over $\CP^1$.
We fix  a divisor ${\bf D}= p_1+\dots+p_n$ on $\C\subset\CP^1$. To 
\[
(v,w)\in\hat{\mathcal X}
\]
one associates a rank-$2$ quasi-parabolic Higgs bundle $(E,\Phi)$ as follows.
The underlying holomorphic bundle is the trivial holomorphic  bundle
\[
E=\CP^1\times\C^2,
\]
equipped at each marked point $p_j\in{\bf D}$ with the quasi-parabolic line
\[
\ell_j:=\mathrm{Span}(v_j)\subset E_{p_j}.
\]
The Higgs field $\Phi$ is the meromorphic $\mathfrak{sl}(2,\C)$-valued $1$-form with
simple poles at the marked points determined by its residues
\[
\Res_{p_j}(\Phi)=(v_j\otimes w_j)_0.
\]
The condition $w_j(v_j)=0$ implies that each residue is nilpotent and preserves the
quasi-parabolic line, so that $\Phi$ is \emph{strongly parabolic}.
Moreover, the equation
\[
\sum_{j=1}^n (v_j\otimes w_j)_0=0
\]
guarantees, by the residue theorem, that $\Phi$ has no further singularity at $\infty\in\CP^1$. 
In order to complete the construction of a parabolic Higgs bundle, the hyperpolygon weight vector 
must additionally satisfy $0 < \alpha_j <\frac{1}{2}$ for $j=1,\dots,n$. 

\begin{proposition}\label{pro:Goma}
{\cite{GoMa}}
If $\alpha$ is small, then a point $(v,w)\in\hat{\mathcal X}$ is $\alpha$-stable in the sense of hyperpolygons if and only if the associated
parabolic Higgs bundle $(E,\Phi)$ is stable with respect to the parabolic weight vector $\alpha$ (see Section~\ref{sec:parahiggs}).
\end{proposition}

Thus, for generic small weights $\alpha$, the hyperpolygon space
$\mathcal X_\alpha^s$ is biholomorphic to a Zariski open subset of the moduli space $\Mcal_{Higgs}(\alpha)$ of  parabolic 
Higgs bundles on $\CP^1$. This biholomorphism preserves the respective 
holomorphic symplectic forms \cite{BiFloGoMa}.
In summary, we have the following result.

\begin{theorem}\label{thm:GM-hyperpolygon-higgs}
{\cite{BiFloGoMa,GoMa}}
Let $\alpha=(\alpha_1,\dots,\alpha_n)\in (0,\frac{1}{2})^n$ be generic and small, and let
$\mathcal X_\alpha$ denote the associated hyperpolygon space.
Let $\mathcal M_{\mathrm{Higgs}}(\alpha)$ be the moduli space of stable 
\emph{strongly parabolic} $\mathfrak{sl}(2,\C)$-Higgs bundles on $(\CP^1,{\bf D})$, and let
$\mathcal H_{\mathrm{Higgs}}(\alpha)\subset\mathcal M_{\mathrm{Higgs}}(\alpha)$ denote
the Zariski-open subset where the underlying holomorphic bundle is trivial.
There exists a natural biholomorphic symplectomorphism
\[
\Psi\colon (\mathcal X_\alpha,I)\xrightarrow{\;\cong\;}
\mathcal H_{\mathrm{Higgs}}(\alpha).
\]
If
$\Omega_{\mathcal X_\alpha}$ denotes the holomorphic symplectic form on
$(\mathcal X_\alpha,I)$, and
$\Omega_{\mathrm{Higgs}}$ denotes the canonical holomorphic symplectic form on
$\mathcal M_{\mathrm{Higgs}}(\alpha)$, then
\[
\Psi^*\Omega_{\mathrm{Higgs}}=\Omega_{\mathcal X_\alpha}.
\]
\end{theorem}

\begin{remark}\label{rem:small-weight}
Since parabolic semi-stability walls \eqref{eq:par-wall}, which intersect  the small-weight region \eqref{eq:polygon-wall} nontrivially,
are necessarily hyperpolygon semi-stability walls, the corresponding parabolic and hyperpolygon genericity conditions coincide in the 
small-weight region and can be unambiguously exchanged. The small-weight condition \eqref{eq:polygon-wall} is necessary and sufficient
for the holomorphic cotangent bundle of the bottom component of the nilpotent cone (see \cite[section 4.2]{BY96} for its construction)
to be entirely contained in the Zariski open subset $\mathcal H_{\mathrm{Higgs}}(\alpha)$ The case $n=4$ is described in detail in \cite{MenesesSIGMA}.
In the case of parabolic weights for which stable parabolic bundles exist, this can also be inferred from \cite[remark 5.4]{AW98}.
This is why we exclusively consider small parabolic weight vectors within the parabolic weight polytope as initial degeneration data 
for scaling transformations towards the vertex $(0,\dots,0)$.
\end{remark}

Although $\mathcal X_\alpha$ and $\mathcal M_{\mathrm{Higgs}}(\alpha)$ are birational 
as complex symplectic manifolds, their hyperk\"ahler metrics are not isometric, since  for arbitrary $\alpha$ both metrics are complete but we have that
\[
\mathcal M_{\mathrm{Higgs}}(\alpha)\backslash \mathcal H_{\mathrm{Higgs}}(\alpha)\neq\emptyset. 
\]
Moreover,  
when $n=4$, the Higgs bundle moduli space carries an ALG metric \cite{FMSW}, whereas
the corresponding hyperpolygon space is known to be ALE.
In general, the hyperk\"ahler metric on $\mathcal X_\alpha$ is defined without any reference 
to the choice of the divisor $\bf{D}$, while the hyperk\"ahler metric on
the Higgs bundle moduli space is  known to depend on the conformal structure of the $n$-punctured sphere.

\section{Construction of twistor lines}\label{sec:construction}
The goal of this section is to construct 
real holomorphic sections of the parabolic Deligne--Hitchin moduli space
$\Mcal_{\mathrm{DH}}^{\mathrm{par}}(t\alpha)$ which converge,
as $t\to 0$, to the hyperpolygon twistor lines arising from the hyperk\"ahler reduction picture.
In particular,  we show that every twistor line of the hyperpolygon space occurs as the
\emph{limiting direction} of a $t$-family of twistor lines for the moduli space 
$\Mcal_{Higgs}(t\alpha)$.

The proof proceeds by formulating a loop--group valued monodromy problem for a family of
$\lambda$--connections and solving it by an implicit function theorem argument.
The resulting real holomorphic sections are then shown to be twistor lines by a global
Iwasawa factorization and a tameness (boundedness) analysis near the punctures.

\subsection{Setup and initial conditions}\label{sec:setupIFT}
We set up necessary loop group background for constructing twistor lines
of  Deligne-Hitchin moduli spaces.

\subsubsection*{Loop algebras and loop groups}
We first fix notations.
Define the loop algebra
\[
\Lambda\mathfrak{sl}(2,\C)
:=
\{\,\xi\colon S^1\to\mathfrak{sl}(2,\C) \mid \xi \text{ is real analytic}\,\}.
\]
Writing $\xi(\lambda)=\sum_{k\in\mathbb Z}\xi_k\lambda^k$, we equip
$\Lambda\mathfrak{sl}(2,\C)$ with the $\ell^1$--norm
\[
\|\xi\|
:=
\sum_{k\in\mathbb Z}|\xi_k|,
\]
where $|\cdot|$ denotes a fixed matrix norm on $\mathfrak{sl}(2,\C)$.
The loop algebra $\Lambda\mathfrak{sl}(2,\C)$ is the Lie algebra of the loop group
\[
\Lambda\mathrm{SL}(2,\C)
:=
\{\, g\colon S^1\to \mathrm{SL}(2,\C)\mid g \text{ is real analytic}\,\}.
\]
We define an anti--holomorphic involution on $\Lambda\mathrm{SL}(2,\C)$ by
\[
g^*(\lambda)
:=
\left(\overline{g(-\bar\lambda^{-1})}^{\,T}\right)^{-1}.
\]
The corresponding unitary loop group is
\[
\Lambda\mathcal U
:=
\{\, g\in \Lambda\mathrm{SL}(2,\C)\mid g^*=g\,\}.
\]
It is a real Banach--Lie group.
We also consider the induced  anti--linear involution
on the loop algebra $\Lambda\mathfrak{sl}(2,\C)$, given by
\[
\xi^*(\lambda)
=
-\overline{\xi(-\bar\lambda^{-1})}^{\,T}.
\]

Clearly, $^*$ is continuous and satisfies
\begin{equation}\label{eq:real_invol} [\xi,\mu]^*=[\xi^*,\mu^*]\end{equation}
and
\[(\sum_k \xi_k\lambda^k)^*=-\sum_k \overline{\xi_{k}}^T \lambda^{-k}.\]

We define the real subspaces
\[
\Lambda\mathfrak U
:=
\{\,\xi\in \Lambda\mathfrak{sl}(2,\C)\mid \xi^*=\xi\,\},
\qquad
\Lambda\mathfrak H
:=
\{\,\xi\in \Lambda\mathfrak{sl}(2,\C)\mid \xi^*=-\xi\,\}.
\]
Then $\Lambda\mathfrak U$ is the Lie algebra of $\Lambda\mathcal U$, and
\[
\Lambda\mathfrak H = i\,\Lambda\mathfrak U.
\]

For $Y\in \Lambda\mathfrak{sl}(2,\C)$ we write
\begin{equation}\label{defuhdeco}
Y
=
Y^{u}+Y^{h}
:=
\tfrac{1}{2}(Y+Y^*)+\tfrac{1}{2}(Y-Y^*)
\in
\Lambda\mathfrak U\oplus\Lambda\mathfrak H,
\end{equation}
and refer to $Y^{u}$ and $Y^{h}$ as the \emph{unitary} and
\emph{Hermitian symmetric} parts of $Y$, respectively.
\begin{example}
Consider $\mathfrak A$ as in \eqref{def:mfA}. Then
\begin{equation}\label{eq:defX}\mathfrak X:=i\mathfrak A\in\Lambda\mathfrak U.\end{equation}
\end{example}

We consider the closed subspaces
\[
\Lambda^0\mathfrak{sl}(2,\C)=\mathfrak{sl}(2,\C),\qquad
\Lambda^+\mathfrak{sl}(2,\C),\qquad
\Lambda^-\mathfrak{sl}(2,\C)
\]
of constant loops, of loops extending holomorphically to the unit disc and vanishing
at $\lambda=0$, and of loops extending holomorphically to the complement of the closed
unit disc and vanishing at $\lambda=\infty$, respectively.
All three are Lie subalgebras of $\Lambda\mathfrak{sl}(2,\C)$.

We further set
\[
\Lambda^{\ge 0}\mathfrak{sl}(2,\C)
:=
\mathfrak{sl}(2,\C)\oplus \Lambda^+\mathfrak{sl}(2,\C),
\qquad
\Lambda^{\ge -1}\mathfrak{sl}(2,\C)
:=
\lambda^{-1}\Lambda^{\ge 0}\mathfrak{sl}(2,\C).
\]
Note that $\Lambda^{\ge -1}\mathfrak{sl}(2,\C)$ is not a Lie subalgebra.

In analogy with the decomposition of the loop algebra, we define the positive loop group
\[
\Lambda^{\geq0}\mathrm{SL}(2,\C)
:=
\{\, g\in \Lambda\mathrm{SL}(2,\C)
\mid g \text{ extends holomorphically to } |\lambda|<1 \,\}.
\]
Its Lie algebra is $\Lambda^{\ge 0}\mathfrak{sl}(2,\C)$.

We further set
\[
\Lambda^{\geq0}_{B}\mathrm{SL}(2,\C)
:=
\{\, g\in \Lambda^{\geq0}\mathrm{SL}(2,\C)
\mid g(0)\in \mathcal B \,\},
\]
where $\mathcal B\subset \mathrm{SL}(2,\C)$ denotes the subgroup of lower triangular
matrices with real entries on the diagonal.

The following  result can be deduced from the Birkhoff factorization in \cite{PressleySegal}. 
\begin{theorem}
There exists
an open dense subset (the \emph{big cell}) of $\Lambda\mathrm{SL}(2,\C)$ such that every
element $g$ in this subset admits a unique factorization
\[
g
=
pu,
\qquad
p\in \Lambda^{\geq0}_{B}\mathrm{SL}(2,\C),
\quad
u\in \Lambda\mathcal U.
\]
The factorization depends real analytically on $g$ wherever it exists.
\end{theorem}

\begin{example}\label{example:factor}\cite{Heller2025}
Let $\alpha>0$, $r\geq0$, and let $\sqrt{1+4r^4}>0$ denote the positive branch of the square root.
Set
\[
A
:=
\alpha
\begin{pmatrix}
 \sqrt{1+4 r^4}  & -\frac{2 r^2}{\lambda } \\
 2 \lambda r^2 & -\sqrt{1+4 r^4}
\end{pmatrix}
\in \Lambda\mathfrak H .
\]
Consider the map
\[
\Psi\colon \R^{>0}\to \Lambda\mathrm{SL}(2,\C),
\qquad
x\longmapsto \exp(-A\log x).
\]
For $x\in(0,1)$ define
\[
B(x)
=
\begin{pmatrix}
 \frac{\sqrt{-\sqrt{1+4 r^4}(x^{4\alpha}-1)+x^{4\alpha}+1}}{\sqrt{2}\,x^{\alpha}} & 0 \\[0.6em]
 -\frac{\sqrt{2}\,r^2(x^{4\alpha}-1)}{x^{\alpha}\sqrt{-\sqrt{1+4 r^4}(x^{4\alpha}-1)+x^{4\alpha}+1}}\,\lambda
 &
 \frac{\sqrt{2}\,x^{\alpha}}{\sqrt{-\sqrt{1+4 r^4}(x^{4\alpha}-1)+x^{4\alpha}+1}}
\end{pmatrix}
\in \Lambda^{\geq0}_{B}\mathrm{SL}(2,\C),
\]
and
\[
F(x)
=
\begin{pmatrix}
 \frac{-\sqrt{1+4 r^4}(x^{2\alpha}-1)+x^{2\alpha}+1}
 {\sqrt{2}\sqrt{-\sqrt{1+4 r^4}(x^{4\alpha}-1)+x^{4\alpha}+1}}
 &
 \frac{\sqrt{2}r^2(x^{2\alpha}-1)}
 {\sqrt{-\sqrt{1+4 r^4}(x^{4\alpha}-1)+x^{4\alpha}+1}}\,\lambda^{-1}
 \\[0.6em]
 \frac{\sqrt{2}r^2(x^{2\alpha}-1)}
 {\sqrt{-\sqrt{1+4 r^4}(x^{4\alpha}-1)+x^{4\alpha}+1}}\,\lambda
 &
 \frac{-\sqrt{1+4 r^4}(x^{2\alpha}-1)+x^{2\alpha}+1}
 {\sqrt{2}\sqrt{-\sqrt{1+4 r^4}(x^{4\alpha}-1)+x^{4\alpha}+1}}
\end{pmatrix}
\in \Lambda\mathcal U,
\]
and observe that
\begin{equation}\label{eq:expIwa}
\Psi(x)=B(x)\,F(x).
\end{equation}
The factorization fails only at the real positive zero 
\[
x_r
=
\left(\frac{\sqrt{4r^4+1}+1}{\sqrt{4r^4+1}-1}\right)^{\tfrac{1}{4\alpha}}
>
1\,
\]
of the denominator of $F$ (or equivalently $B$), ensuring existence for all $x\in(0,1).$
Note that \eqref{eq:expIwa} extends  to $r=0$, for the constant unitary part $F_{r=0}=\mathrm{Id}.$

If instead one considers
\[
\tilde A
:=
\alpha
\begin{pmatrix}
 -\sqrt{1+4 r^4}  & -\frac{2 r^2}{\lambda } \\
 2 \lambda r^2 & \sqrt{1+4 r^4}
\end{pmatrix}
\in \Lambda\mathfrak H,
\qquad
\tilde\Psi(x)=\exp(-\tilde A\log x),
\]
then the corresponding Iwasawa factorization fails at
$\tilde x_r=1/x_r\in(0,1)$.
In fact, $\tilde B$ and $\tilde F$ are obtained from $B$ and $F$ by replacing
$\sqrt{1+4r^4}$ with $-\sqrt{1+4r^4}$ throughout.
\end{example}

\subsubsection*{The monodromy problem}

Consider the $n$--punctured sphere
\[
\CP^1\setminus\{p_1,\dots,p_n\}.
\]
Without loss of generality, we assume that $p_j\in\C\subset\CP^1$ for all $j$.
Fix simple closed loops $\gamma_j$ in $\C\setminus\{p_1,\dots,p_n\}$,
based at a point $q$, winding once positively around $p_j$, such that
\[
\gamma_n*\dots*\gamma_1=1
\quad\text{in}\quad
\pi_1(\CP^1\setminus\{p_1,\dots,p_n\},q).
\]

We consider real analytic families of $\Lambda\mathfrak{sl}(2,\C)$--connections
\[
\nabla_t = d+\xi(t),
\qquad t\in[0,\varepsilon),
\]
with $\xi(0)=0$ and
\[
\xi(t)\in H^0\!\left(\CP^1,
K(p_1+\dots+p_n)\otimes \Lambda^{\geq-1}\mathfrak{sl}(2,\C)\right).
\]

Let
\[
\rho_t\colon \pi_1(\CP^1\setminus\{p_1,\dots,p_n\},q)\to \Lambda\mathrm{SL}(2,\C)
\]
be its monodromy representation, and set $M_j(t):=\rho_t(\gamma_j)$.
We define the logarithmic derivatives of the monodromy by
\[
B_j(t)
:=
M_j(t)^{-1}M_j'(t)
\in \Lambda\mathfrak{sl}(2,\C),
\qquad j=1,\dots,n.
\]

Fix a small generic weight vector $\alpha=(\alpha_1,\dots,\alpha_n)$, and a stable strongly parabolic Higgs bundle with 
Higgs field $\Phi\neq0$ whose underlying holomorphic bundle is trivial.
In particular, the quasi-parabolic lines $\ell_1,\dots,\ell_n$ are fixed.
Our goal is to construct a family $\xi(t)$ as above such that the following
conditions are satisfied:
\begin{itemize}
\item[(i)]
$B_j(t)\in\Lambda\mathfrak U$ for all $j=1,\dots,n$;
\item[(ii)]
$\xi_{-1}(t)= \Res_{\lambda=0}\xi(t)=t\Phi$,   with fixed quasi-parabolic lines $\ell_j$;
\item[(iii)]
$\det\!\left(\tfrac{1}{t}\Res_{p_j}\xi(t)\right)=-\alpha_j^2$
for all $j=1,\dots,n$ and all $\lambda\in\C^*$.
\end{itemize}
We refer to this system of conditions as the \emph{monodromy problem}.
\begin{remark}\label{rem:quasipara}
Using $\xi=\sum_{k\geq-1}\xi^{(k)}\lambda^k$
together with (iii), the eigenline of $\Res_{p_j}\xi(t)$ with respect to the eigenvalue $t\alpha$
extends to $\lambda=0.$ We call this the induced quasiparabolic line of $d+\xi(t)$ at $\lambda=0$.

The quasi-parabolic line $\ell_j$ at $p_j$ induced by $d+\xi(t)$ at $\lambda=0$ is given by
\[
\ell_j = \ker\bigl(\operatorname{Res}_{p_j}\Phi\bigr),
\]
whenever $\operatorname{Res}_{p_j}\Phi \neq 0$.
In this situation, the quasi-parabolic line $\ell_j$ is independent of $t$, since
$
\xi_{-1}(t) = t\Phi .
$
\end{remark}

\subsection{Solving the monodromy problem}

We first address condition~(ii) in the monodromy problem. Let $(v,w)\in\mathcal X^\alpha$ correspond to the strongly
parabolic Higgs field $\Phi$ via Proposition~\ref{pro:Goma}.
For $j=1,\dots,n$ let $\mathfrak A_j=\mathfrak A(v_j,w_j)$ be as in \eqref{eq:mafraA}.
By Proposition~\ref{prop:HK-KN} and stability of $\Phi$, we may assume (after
$\mathrm{SL}(2,\C)$--conjugation) that
all moment map conditions for the hyperpolygon are solved, i.e.
\[
\sum_{j=1}^n \mathfrak A_j = 0,
\]
and define
\begin{equation}\label{eq:xicentral}
\underline\xi
:=
\sum_{j=1}^n \mathfrak A_j\,\frac{dz}{z-p_j}.
\end{equation}
Consider the ansatz
\begin{equation}\label{eq:defxit}
\xi(t)
=
t\bigl(\underline\xi+\hat\xi(t)\bigr),
\end{equation}
where $\hat\xi(t)$ is $\Lambda^{\ge0}\mathfrak{sl}(2,\C)$--valued and satisfies
$\hat\xi(0)=0$.
 In particular, this means that the Higgs field of $\xi(t)$ 
(determined by the residue at $\lambda=0$) 
is  given by $\underline\xi$.

Write
\begin{equation}\label{eq:defhatxi}
\hat\xi(t)
=
\sum_{j=1}^n P_j(t)\,\frac{dz}{z-p_j},
\qquad
P_j(t)\in \Lambda^{\ge0}\mathfrak{sl}(2,\C).
\end{equation}
For technical reasons, we initially allow $\hat\xi(t)$ to have a simple pole at
$z=\infty\in\CP^1$, i.e. we do not impose the condition $\sum_j P_j(t)=0$.
Although this might introduce a pole at infinity, the monodromy along the loops $\gamma_j$ (which avoid $\infty$)
is still well defined and
can be computed.

We next observe that the monodromy problem is solved at $t=0.$
\begin{lemma}\label{lem:IFTt0}
Let $\xi(t)$ be given by \eqref{eq:defxit}, where $\underline\xi$ is defined in
\eqref{eq:xicentral} and $\hat\xi(0)=0$.
Then conditions \emph{(i)}, \emph{(ii)}, and \emph{(iii)} of the monodromy problem
are satisfied at $t=0$.
\end{lemma}

\begin{proof}
Condition~\emph{(ii)} holds by construction of $\underline\xi$, and
condition~\emph{(iii)} follows
from \eqref{eq:mafraA}
(by analytic continuation to $t=0$).
To verify~\emph{(i)}, let $\Psi_t$ be the solution of
\[
d\Psi_t+\xi(t)\Psi_t=0,
\qquad
\Psi_t(q)=\mathrm{Id}.
\]
Then, for small $t$ we expand using the assumption $\hat\xi(0)=0$
\[
\Psi_t(z)
=
\mathrm{Id}
-
t\int_q^z \underline\xi
+
O(t^2).
\]
By the residue theorem  we obtain
\[
B_j(0)=-2\pi i\,\mathfrak A_j\in\Lambda\mathfrak U,
\]
which proves the claim.
\end{proof}

We now study the monodromy problem for $t\neq 0$ sufficiently small.
The strategy is to view conditions~\emph{(i)} and~\emph{(iii)} as a system of equations
for the unknown correction term $\hat\xi(t)$ and to solve this system by an implicit
function theorem argument between the closed subspaces
$\Lambda\mathfrak H\cap \mathcal K_{\mathfrak X}$ and $\Lambda^{\geq0}\mathfrak{sl}(2,\C)$
 of  the Banach space
$\Lambda\mathfrak{sl}(2,\C)$,  (see Appendix~\ref{appendix}).

Condition~\emph{(i)} of the monodromy problem requires that
\begin{equation}\label{eq:difres}
B_j(t):=M_j(t)^{-1}M_j'(t)\in\Lambda\mathfrak U,
\quad j=1,\dots,n.
\end{equation}
Equivalently, this condition can be written as
\begin{equation}\label{eq:unitary-condition}
\bigl(B_j(t)\bigr)^* = B_j(t),
\quad j=1,\dots,n.
\end{equation}

Similarly, condition~\emph{(iii)} prescribes the determinant of the residues of $\xi(t)$.
Using \eqref{eq:defxit} and \eqref{eq:defhatxi}, we compute
\[
\frac{1}{t}\Res_{p_j}\xi(t)
=
\mathfrak A_j + P_j(t),
\]
where $\mathfrak A_j$ are the residues of $\underline\xi$.
Thus condition~\emph{(iii)} is equivalent to
\begin{equation}\label{eq:det-condition}
\det\bigl(\mathfrak A_j + P_j(t)\bigr) = -\alpha_j^2,
\quad j=1,\dots,n,
\end{equation}
for all $t$.

Equations \eqref{eq:unitary-condition} and \eqref{eq:det-condition} form a coupled
system of nonlinear equations for the unknowns $P_j(t)$.
By Lemma~\ref{lem:IFTt0}, this system is satisfied at $t=0$.
We now linearize it at $t=0$.

Let $\delta P_j\in \Lambda^{\geq 0}\mathfrak{sl}(2,\C)$ denote variations of $P_j$ at $t=0$.
Differentiating \eqref{eq:difres} at $t=0$ and using the computation of Lemma \ref{lem:IFTt0} yields
\[
\delta B_j = -2\pi i\,\delta P_j .
\]
Similarly, differentiating \eqref{eq:det-condition} gives
\begin{equation}\label{eq:linear-det}
\tr\bigl(\mathfrak A_j\,\delta P_j\bigr)=0,
\quad j=1,\dots,n.
\end{equation}
Consider the closed subspaces  
\begin{equation*}\begin{split}
\mathcal K_{j}:&=\ker\left(\,\delta P_j\in \Lambda\mathfrak{sl}(2,\C)\mapsto 
\tr(\mathfrak A_j \delta P_j)\in\Lambda \C \,\right)\\
\mathcal K^{\geq0}_{j}:&=\ker\left(\,\delta P_j\in \Lambda^{\geq 0}\mathfrak{sl}(2,\C)\mapsto 
\tr(\mathfrak A_j \delta P_j)\in\Lambda^{\geq-1}\C \,\right)\end{split}\end{equation*}
of $\Lambda^{\geq 0}\mathfrak{sl}(2,\C)$, and
the real projection
\[\pi^{\mathfrak{H}}\colon  \Lambda\mathfrak{sl}(2,\C)\to \Lambda\mathfrak H.\]
From \eqref{def:mfA} and Subsection \ref{sec:hptwistor} we know that the eigenline of $\mathfrak A_j$ with respect
to the eigenvalue $\alpha_j>0$ extends to $\lambda=0,$ and coincides with $\ell_j.$ We define
\[\mathring {\mathcal K}^{\geq0}_{j}:=\{\,\delta P_j\in\mathcal K^{\geq0}_{j}\mid \delta P_j^{(0)}(\ell_j)\subset\ell_j\}\]
as those elements in $\mathcal K^{\geq0}_{j}$ which preserve $\ell_j.$ Note that for
$\Res_{p_j}\Phi\neq0$ we have
\[\mathring {\mathcal K}^{\geq0}_{j}={\mathcal K}^{\geq0}_{j}\]
by Remark  \ref{rem:quasipara}.

The key point is that the continuous linear map
\begin{equation}\label{eq:linIFT}
(\delta P_1,\dots,\delta P_n)\in (\mathring {\mathcal K}^{\geq0}_{1},\dots,\mathring {\mathcal K}^{\geq0}_{n})
\longmapsto
\bigl(\pi^{\mathfrak{H}}\delta B_1,\dots,\pi^{\mathfrak{H}}\delta B_n\bigr)\in \Lambda\mathfrak H^n
\cap(\mathcal K_1,\dots,\mathcal K_n)\end{equation}
is bijective (as $\beta_j\neq0$).
This follows from Proposition~\ref{prop:solutionatpj} in the Appendix~\ref{appendix},
applied pointwise to each residue (after appropriate $\mathrm{SU}(2)$--conjugation).


As a consequence, the linearization of the full monodromy problem at $t=0$ is an
isomorphism between the relevant Banach spaces.
The implicit function theorem therefore applies and yields the following result.
\begin{lemma}\label{lem:monodromy-ift}
Let $\underline\xi$ be given by \eqref{eq:xicentral}.
There exists $\epsilon>0$ such that for all $t\in(-\epsilon,\epsilon)$ there is a
\emph{unique} real analytic family
\[
\hat\xi(t)=\sum_{j=1}^n P_j(t)\,\frac{dz}{z-p_j},
\qquad
P_j(t)\in \Lambda^{\geq0}\mathfrak{sl}(2,\C),
\]
with $\hat\xi(0)=0$, satisfying conditions~\emph{(i)},~\emph{(ii)}, and~\emph{(iii)}
of the monodromy problem.
The maps $t\mapsto P_j(t)$ depend real analytically on $t$ and on the
initial data $\underline\xi$.
\end{lemma}
\noindent

To obtain real holomorphic sections, the remaining condition is $\sum_j P_j(t)=0$. This will be arranged in the next section.

\subsection{Construction of real holomorphic sections}
We now show that, for arbitrary initial data $\underline\xi$, 
the monodromy problem can be modified in such a way
that the correction terms $P_j(t)$ satisfy an additional normalization condition.
More precisely, we will arrange that
\begin{equation}\label{eq:sum-commutator-condition}
\sum_j P_j(t)
\in
\mathfrak{su}(2)\oplus \Lambda^{+}\mathfrak{sl}(2,\C)
\subset
\Lambda^{\geq0}\mathfrak{sl}(2,\C)
\end{equation}
for $t$ sufficiently small. We then show that this condition implies $\sum_j P_j(t)=0.$

To achieve this, we allow the Higgs field to vary with $t$ within its
$\mathrm{SL}(2,\C)$--orbit, with the variation chosen orthogonally to
$\mathfrak{su}(2)$ at $t=0$.
Equivalently, we allow a $t$--dependent conjugation of the residues
$\mathfrak A_j^{(-1)}$ by elements of $\mathrm{SL}(2,\C)$.
Note that conjugation by elements of $\mathrm{SL}(2,\C)\setminus \mathrm{SU}(2)$
does not change the gauge class of the underlying Higgs field, but modifies its representative
\[
\sum_j\mathfrak A_j^{(-1)}\frac{dz}{z-p_j},
\]
which is regular at $z=\infty$, in such a way that the matrix 
\[
0\neq \sum_j {\mathfrak A}_j \in i\,\mathfrak{su}(2).
\]

This observation indicates that the normalization condition
\eqref{eq:sum-commutator-condition} is attainable and motivates the following
modified monodromy problem, where only condition~\emph{(ii)} is replaced by
\begin{itemize}
\item[(iia)]
for all $t$ sufficiently small, the residue $\Res_{\lambda=0}\xi(t)$ lies in the
$\mathrm{SL}(2,\C)$--orbit of the Higgs field $t\Phi$, and 
$\Res_{p_j}\xi(t)$ induces the  corresponding $\mathrm{SL}(2,\C)$--conjugated quasi-parabolic line at $\lambda=0$.
\end{itemize}
As in Remark~\ref{rem:quasipara} above, the second condition in (iia) follows from the first condition
in (iia) if $\Res_{p_j}\Phi\neq0.$

\begin{lemma}\label{lem:modified-monodromy}
The modified monodromy problem \emph{(i)}, \emph{(iia)}, and \emph{(iii)} admits a solution such that
\begin{equation}\label{eq:iftaddon}
\sum_j P_j
\in
\mathfrak{su}(2)\oplus \Lambda^{+}\mathfrak{sl}(2,\C)
\subset
\Lambda^{\geq0}\mathfrak{sl}(2,\C).
\end{equation}
The solution is unique up to (time--dependent) overall $\mathrm{SU}(2)$--conjugation, and can be chosen
to depend real analytically on $t$ and on $\underline\xi$.
\end{lemma}

\begin{proof}
We must show that the infinitesimal action of $i\,\mathfrak{su}(2)$ on the Higgs field
produces precisely the $i\,\mathfrak{su}(2)$--part of the residue
$\sum_j P_j(t)$ at $z=\infty$.

Recall that the Higgs field and the zeroth--order components of $\underline\xi$ are given by
\begin{equation}\label{eq:higgs-components}
\begin{aligned}
\Phi &= \underline\xi^{(-1)}
      = \sum_j \nu_{\C}(v_j,w_j)\,\frac{dz}{z-p_j},\\
\underline\xi^{(0)} &= -2i\sum_j \nu_{I}(v_j,w_j)\,\frac{dz}{z-p_j}
\in i\,\mathfrak{su}(2).
\end{aligned}
\end{equation}
The infinitesimal action of $\eta\in i\,\mathfrak{su}(2)$ on $\Phi$ and $\Phi^{*}$ is given by
the commutator $[\eta,\cdot]$.
Since both $\Phi$ and $\Phi^{*}$ have vanishing residue at $z=\infty$
(by the assumption $p_1,\dots,p_n\in\C$),
this action does not produce any residue in the $\lambda^{-1}$-- or $\lambda$--terms.

It therefore suffices to analyze the $\lambda^{0}$--component $\underline\xi^{(0)}$.
Let $v_j\in\C^2$, $w_j\in(\C^2)^*$, and let $\eta\in i\,\mathfrak{su}(2)$.
A direct computation shows
\begin{equation}\label{eq:infinitesimal-action}
\begin{aligned}
\eta.(v_j\otimes v_j^* - w_j^*\otimes w_j)_0
&=
(\eta v_j\otimes v_j^* + v_j\otimes v_j^*\eta^\dagger
 + \eta^\dagger w_j^*\otimes w_j + w_j^*\otimes w_j\eta)_0\\
&=
(\eta v_j\otimes v_j^* + v_j\otimes v_j^*\eta
 + \eta w_j^*\otimes w_j + w_j^*\otimes w_j\eta)_0\\
&=
(|v_j|^2 + |w_j|^2)\,\eta,
\end{aligned}
\end{equation}
where $\eta^\dagger=\bar\eta^T$.
Note that  $\Res_{p_j}\Phi=0$  is equivalent to $w_j=0$, 
but $v_j\neq0$ spans the quasi-parabolic line $\ell_j$ as in Subsection~\ref{rem:Scaling}.
Thus
$
\sum_j (|v_j|^2 + |w_j|^2)\neq0,
$
and  the infinitesimal action of $i\,\mathfrak{su}(2)$ on
$\underline\xi^{(0)}$ is a linear isomorphism onto $i\,\mathfrak{su}(2)$.
Hence, by the implicit function theorem and the previous monodromy construction,
we can uniquely solve for condition~\emph{(iia)}.
Clearly, (infinitesimal) conjugation does not change conditions \emph{(i)} and \emph{(iii)}.
\end{proof}


\begin{lemma}\label{lem:sum-P-zero}
Consider the solution $\hat\xi(t)$ of the modified monodromy problem
in Lemma~\ref{lem:modified-monodromy} such that \eqref{eq:iftaddon} holds.
Then
\[
\sum_j P_j(t)=0.
\]
\end{lemma}

\begin{proof}
Consider the monodromy of $d+\xi(t)$ based at $q$ as a representation
\[
\pi_1(\CP^1\setminus\{p_1,\dots,p_n,\infty\},q)\to\Lambda\mathrm{SL}(2,\C)
\]
of the $(n+1)$--punctured sphere.
Let $\gamma_\infty$ be a simple loop winding once around $z=\infty$ and satisfying
\[
\gamma_\infty * \gamma_n * \dots * \gamma_1 = 1.
\]

We compute the corresponding logarithmic derivative $B_\infty(t)$ in two ways.
On the one hand, it is given by the logarithmic derivative of
\[
M_\infty(t)
=
\bigl(M_n(t)\, M_{n-1}(t)\dots M_1(t)\bigr)^{-1}
\in
\Lambda\mathrm{SL}(2,\C).
\]
Since $B_j(t)\in\Lambda\mathfrak U$ and $M_j(0)=\mathrm{Id}$ for $j=1,\dots,n$,
we have $M_j(t)\in\Lambda\mathcal U$ for $t$ sufficiently small.
Hence $M_\infty(t)\in\Lambda\mathcal U$, and therefore
\[
B_\infty(t)\in\Lambda\mathfrak U.
\]

On the other hand, we expand the residue at $z=\infty$ in terms of $t$ at $t=0$:
\[
\sum_j P_j(t)= \mathcal C^{(k)}\, t^{k}+o(t^{k}),
\]
for some $k\geq1$ and some
\[
\mathcal C^{(k)}\in \mathfrak{su}(2)\oplus \Lambda^{+}\mathfrak{sl}(2,\C)
\]
by Lemma~\ref{lem:modified-monodromy}.
Arguing as in the proof of Lemma~\ref{lem:IFTt0}
(by expanding the parallel frame $\Psi_t$ solving
$d\Psi_t+\xi(t)\Psi_t=0$ with $\Psi_t(q)=\mathrm{Id}$ in $t$),
we obtain (for some $k\geq1$)
\[
B_\infty(t)=2\pi i\, \mathcal C^{(k)}\, t^{k}+o(t^{k}).
\]

Thus
\[
2\pi i\, \mathcal C^{(k)}\in 2\pi i\,
\bigl(\mathfrak{su}(2)\oplus \Lambda^{+}\mathfrak{sl}(2,\C)\bigr)
\cap \Lambda\mathfrak U
=
\{0\},
\]
and since $B_\infty(t)$ is real analytic in $t$, we conclude that
$B_\infty(t)=0$ for all $t$ sufficiently small.
Equivalently,
\[
\sum_j P_j(t)=0.
\]
This completes the proof.
\end{proof}

\begin{remark}[Perturbed moment map]
Let 
\[\pi^\mathfrak h\colon \Lambda^{\geq0}\mathfrak{sl}(2,\C)\to i\mathfrak{su}(2);\; \sum_{k\geq0}x_k\lambda^k\mapsto \tfrac{1}{2}(x_0+\bar x_0^T)\]
be the projection onto $i\mathfrak{su}(2)$. 
Identify strongly parabolic Higgs bundles with twistor lines \eqref{eq:xicentral} at $t=0.$
Consider
the map 
\[\nu_I(t)\colon \mathcal M_{Higgs}(t\alpha)\to\mathfrak{su}(2); 
\underline\xi\mapsto \,-i\,\pi^\mathfrak h\left(\sum_j P_j(t)\right)\]
where $P_1(t),\dots,P_n(t)$ are 
obtained from Lemma \ref{lem:monodromy-ift}.
By Lemma \ref{lem:modified-monodromy} and 
Lemma \ref{lem:sum-P-zero} (together with the uniqueness of the implicit function theorem) the map
 $\nu_I(t)$ is equivariant and can be seen as the perturbation of the real moment map $\nu_I$ for the $\mathrm{SU}(2)$-action, see Subsection \ref{sec:hptwistor} and
Remark \ref{rem:twisted}.
\end{remark}

\begin{theorem}\label{thm:real-holomorphic-section-existence}
Let $\alpha$ be small and generic.
Let $(\Pcal,\alpha,\Phi)$ be a stable strongly parabolic Higgs pair on
$(\CP^1,p_1+\dots+p_n)$ with trivial underlying holomorphic bundle, and let $\underline\xi$ be the associated
initial $\lambda$--dependent $1$--form defined in \eqref{eq:xicentral}.
Then there exists $\epsilon>0$ such that for all $t\in(-\epsilon,\epsilon)$ there is a
real analytic family
\[
\hat\xi(t)\in
H^0\!\left(
\CP^1,
K(p_1+\dots+p_n)\otimes\Lambda^{\geq0}\mathfrak{sl}(2,\C)
\right),
\qquad \hat\xi(0)=0,
\]
with the following properties:
\begin{itemize}
\item[(i)]
The connection
\[
d+t\underline\xi+t\hat\xi(t)
\]
represents the Higgs pair $(\Pcal,t\alpha,t\Phi)$ at $\lambda=0$.

\item[(ii)]
The monodromy of $d+t\underline\xi+t\hat\xi(t)$, based at $q$, satisfies
\[
B_j(t)=M_j^{-1}(t)M_j'(t)\in \Lambda\mathfrak U
\qquad \text{for all } j=1,\dots,n,
\]
that is, the monodromy representation is $\Lambda\mathcal U$--valued.

\item[(iii)]
The residues satisfy
\[
\det\!\left(\tfrac{1}{t}\Res_{p_j}\bigl(t\underline\xi+t\hat\xi(t)\bigr)\right)
=-\alpha_j^2
\qquad\text{for all } j=1,\dots,n.
\]
\end{itemize}

The family $\hat\xi(t)$ is unique up to $t$--dependent conjugation by
$\mathrm{SU}(2)$.
Moreover, $\hat\xi(t)$ can be chosen to depend locally real analytically on $t$
and on the Higgs data $(\Phi,\alpha)$; in particular, the existence interval
can be chosen locally uniformly in $(\Phi,\alpha)$.
\end{theorem}

\begin{proof}
Locally in the moduli space, we can normalize the Higgs field with respect to the
$\mathrm{SU}(2)$--freedom, for instance by fixing suitable non--vanishing residues.
Under such a local normalization, the form $\underline\xi$ depends real analytically
on the Higgs field and on the weight vector $\alpha$, as long as $\alpha$ does not
cross stability walls.
Moreover, under this normalization, the solution of the modified monodromy problem
depends real analytically on $t$ and on $\underline\xi$, which proves the theorem.
\end{proof}

This completes the analytic construction of real holomorphic sections; the next task is to identify these sections as genuine twistor lines.
In view of the hyperpolygon interpretation discussed after Theorem \ref{thm:twistor-hyperpolygon-orbits}, the result can be viewed as a deformation of hyperpolygon twistor lines for small weights.

\subsection{Twistor lines}\label{sec:constwistor}
In this subsection we show that the $\Lambda^{\geq-1}\mathfrak{sl}(2,\C)$--connections
$d+t\underline\xi+t\hat\xi(t)$ constructed in
Theorem~\ref{thm:real-holomorphic-section-existence}
give rise to twistor lines in the parabolic Deligne--Hitchin moduli space
corresponding to the parabolic weight vector $t\alpha$.
We first show that $d+t\underline\xi+t\hat\xi(t)$ defines a real holomorphic section.

\begin{lemma}\label{lem:realholosec}
Let $\alpha$ be a generic weight vector.
Let $t>0$ be sufficiently small, and let $d+t\underline\xi+t\hat\xi(t)$ be as constructed
in Theorem~\ref{thm:real-holomorphic-section-existence}.
Then there exists a real holomorphic section $s$ of the parabolic Deligne--Hitchin
moduli space $\mathcal M_{\mathrm{DH}}^{\mathrm{par}}(\alpha)$ such that, for all
$\lambda$ with $|\lambda|\leq1$, the trivial holomorphic bundle together with the
logarithmic $\lambda$--connection
\[
D_\lambda:=\lambda\bigl(\partial+t\underline\xi+t\hat\xi(t)\bigr)
\]
represents a point
$s(\lambda)\in \pi^{-1}(\lambda)\subset\mathcal M_{\mathrm{DH}}^{\mathrm{par}}(\alpha)$.
\end{lemma}

\begin{proof}
First note that, by construction, $\lambda t(\underline\xi+\hat\xi(t))$ extends
holomorphically to $\lambda=0$ and gives rise to a stable strongly parabolic Higgs
field by condition~\emph{(i)} in
Theorem~\ref{thm:real-holomorphic-section-existence}.
Moreover, by condition~\emph{(iii)} in the same theorem,
$d+t\underline\xi+t\hat\xi(t)$ defines a parabolic $\lambda$--connection
corresponding to the weight vector $\alpha$ for all $\lambda$ in an open neighbourhood
$U$ of the closed unit disc.
We therefore obtain a holomorphic section $s$ of
$\mathcal M_{\mathrm{DH}}^{\mathrm{par}}(\alpha)$ over $U$.

By condition~\emph{(ii)}, the monodromy of $d+t\underline\xi+t\hat\xi(t)$ based at $q$
is $\Lambda\mathcal U$--valued.
In particular, for all $\lambda\in S^1$ we have
\[
\tau\bigl(s(-\bar\lambda^{-1})\bigr)=s(\lambda),
\]
see Section~\ref{sec:pdhms}.
Hence, by Schwarz reflection, $s$ extends to a global holomorphic section of
$\mathcal M_{\mathrm{DH}}^{\mathrm{par}}(\alpha)\to\CP^1$ which is real by construction.
\end{proof}

The real holomorphic sections constructed in Lemma~\ref{lem:realholosec} are in fact
twistor lines. To see this, we need a few preparatory results.

\subsubsection*{Harmonic maps}

The relation between harmonic maps to hyperbolic $3$--space and solutions of
Hitchin's self--duality equations is classical \cite{Donald}.
For further details, see \cite[Section~3.2]{OSWW} or \cite[Section~2.5]{BBDH}.
Let $\Sigma$ be a Riemann surface, and let $h$ be a Hermitian metric on a holomorphic
rank~$2$ vector bundle $V\to\Sigma$ with trivial determinant.
Let $D_h$ be the Chern connection associated to $h$ and $V$, and let $\Phi^*$
denote the adjoint of a Higgs field $\Phi$ with respect to $h$.
Assume that $D_h+\Phi+\Phi^*$ is flat, i.e.\ that $(V,h,\Phi)$ is a solution of
Hitchin's self--duality equations.
In particular, $D_h-\Phi-\Phi^*$ is also flat, and the corresponding monodromy
representations are complex--conjugate contragredient representations of each other.

Consider parallel frames $F_\pm$ with respect to $D_h\pm\Phi\pm\Phi^*$, normalized
by $F_\pm(q)=\mathrm{Id}$ at a fixed base point $q\in\Sigma$.
Then
\begin{equation}\label{eq:symbobenko}
f
=
F_-^{-1}F_+
=
\bar F_+^{\,T} F_+
\colon
\widetilde\Sigma\to \mathbb H^3
\cong
\{\,f\in\mathrm{SL}(2,\C)\cap i\mathfrak{su}(2)\mid f>0\,\}
\end{equation}
defines an equivariant harmonic map to hyperbolic $3$--space.
Here, hyperbolic $3$--space is identified with the space of positive definite
Hermitian matrices of determinant~$1$.


\subsubsection{Loop group factorization method}

The loop group factorization method was introduced for harmonic maps from simply connected surfaces into compact
symmetric spaces by \cite{DPW}.
See \cite{HellerHellerTraizet2025} for details in the case of equivariant
harmonic maps to $\mathbb H^3$ from surfaces with non-trivial topology.

Consider $\eta\in H^0(\Sigma, K\otimes \Lambda^{\geq-1}\mathfrak{sl}(2,\C))$, and let
$\Psi$ be a solution of
\[
d\Psi+\eta\Psi=0,
\qquad
\Psi(q)=\mathrm{Id}.
\]
In general, $\Psi$ may have non--trivial monodromy and is therefore well defined
only on the universal cover $\widetilde\Sigma$.
Assume that the Iwasawa decomposition
\[
\Psi(z)=B(z)F(z),
\qquad
B(z)\in \Lambda^{+}_{B}\mathrm{SL}(2,\C),
\quad
F(z)\in \Lambda\mathcal U,
\]
exists for all $z\in\widetilde\Sigma$.
Furthermore, assume that the monodromy of $\Psi$ based at $q$ is contained in
$\Lambda\mathcal U$.
Then the uniqueness of the Iwasawa decomposition implies that the monodromy of
$B$ based at $q$ (and hence everywhere) is trivial.

Moreover,
\[
(d+\eta).\,B
=
dF.\, F^{-1}
\in \Omega^1(\Sigma,\Lambda\mathfrak U)
\]
is well-defined on $\Sigma$.
Using that $B$ extends holomorphically to $\lambda=0$, and comparing the left--hand
side and the right--hand side of the above equation, shows that
$(d+\eta).\,B$ has the form of the associated family of flat connections
$\nabla_\lambda$ as in \eqref{eq:assocfam}.

Finally, $F_\pm:=F(\lambda=\pm1)$ are parallel frames for $\nabla_{\pm1}$, and
\eqref{eq:symbobenko} yields an (equivariant) harmonic map to hyperbolic $3$--space.

\subsubsection{Tameness}

Let $\alpha>0$ and fix an $\mathrm{SL}(2,\C)$--frame $(e_1,e_2)$.
The Hermitian metric
\[
h_{\mathrm{mod}}(z)
=
\begin{pmatrix}
 (z\bar z)^\alpha & 0 \\
 0 & (z\bar z)^{-\alpha}
\end{pmatrix}
\]
is called the \emph{model metric} with respect to the weight $\alpha$ and the
complex line $\ell:=\C e_1$.

Following \cite{Sim2}, a map
\[
h\colon \mathbb D\setminus\{0\}\longrightarrow \mathbb H^3
\]
is called \emph{tame} at $0$ with respect to the weight $\alpha>0$ and the
complex line $\ell\subset\C^2$ if its hyperbolic distance to the
corresponding model metric is bounded, i.e.\ if there exists $C>0$ such that
\[
d^{\mathrm{hyp}}\bigl(h(z),h_{\mathrm{mod}}(z)\bigr)< C
\qquad
\text{for all } z\in \mathbb D\setminus\{0\}.
\]


\begin{remark}
Note that
tameness of $h$ is not measured with respect to a parallel frame, but rather with respect
to a frame of the underlying holomorphic bundle.
However, in the case of a solution of Hitchin's equations, we will also call the induced
equivariant harmonic map $f$ a \emph{tame harmonic map}.
\end{remark}

\begin{proposition}\label{prop:tame-bounded-cover}
Let $(E,\Phi)$ be a rank~$2$ strongly parabolic Higgs bundle on the disc $\mathbb D$ with parabolic
weight $0<\alpha<\tfrac{1}{2}$ at $p \in \mathbb D$. Let $h$ be a Hermitian metric on
$E_{\mid \mathbb D\setminus \{p\}}$ solving Hitchin's equations.
Denote by $f$ the equivariant harmonic map
 associated to the solution of the self--duality equations.
Then $h$ is tame at $p$ if and only if $f$ is bounded near $p$.
\end{proposition}

\begin{proof}
This follows from \cite[Section~7]{Sim2}, since the residue of the Higgs field is
nilpotent.
For rational weights, this also follows from \cite{Nas}, as the harmonic map lifts
to an equivariant harmonic map with rotational symmetry on the covering.
\end{proof}

\begin{lemma}[Model solution at one puncture]\label{lem:analysisononeleg}
Let $0<\alpha<\tfrac{1}{2}$, $\beta\in\C$, and set $\mu=i\sqrt{\alpha^2+\beta\bar\beta}$.
Consider
\begin{equation}\label{eq:mfx}
\mathfrak X
=
\lambda^{-1}
\begin{pmatrix}
0 & \beta \\
0 & 0
\end{pmatrix}
+
\begin{pmatrix}
\mu & 0 \\
0 & -\mu
\end{pmatrix}
+
\lambda
\begin{pmatrix}
0 & 0 \\
\bar\beta & 0
\end{pmatrix}
\in
\Lambda\mathcal U,
\end{equation}
and define $\eta=A\,\tfrac{dz}{z}$ with $A:=-i\mathfrak X$.
Then the monodromy of $\eta$ has values in $\Lambda\mathcal U$, and the parallel
frame $\Psi$ of $d+\eta$ admits an Iwasawa decomposition for all
$z\in\mathbb D\setminus\{0\}$.
The corresponding equivariant harmonic map is tame, and constant if $\beta=0$.
\end{lemma}

\begin{proof}
On the universal cover, a solution of the parallel frame equation is given by
\[
\Psi(z)=\exp(-A\log z),
\]
with monodromy around $z=0$ given by
\[
\exp(2\pi i A)\in\Lambda\mathcal U.
\]
That the Iwasawa decomposition exists for all $z\in\mathbb D\setminus\{0\}$ follows
from Example~\ref{example:factor} as follows. 
It works completely analogous to the proof of Proposition~19 in \cite{CHHT}.
Using the $S^1$--action (by rotation in the loop parameter $\lambda$), we may
assume without loss of generality that $\beta>0$.
Writing $z=x e^{i\varphi}$ with $x>0$ and $\varphi\in\R$, we obtain the Iwasawa decomposition of $\Psi$ to be
\begin{equation}\label{eq:spherefac}
\Psi(z)
=
\exp(-A\log x)\exp(-iA\varphi)
=
B(x)\bigl(F(x)\exp(-iA\varphi)\bigr),
\end{equation}
for $F(x)\in\Lambda\mathcal U$,
$B(x)\in\Lambda^{\geq0}\mathrm{SL}(2,\C)$ are given in Example \ref{example:factor}, and $\exp(-iA\varphi)\in\Lambda\mathcal U$.

Using this explicit factorization \eqref{eq:spherefac} together with the
Sym--Bobenko formula \eqref{eq:symbobenko}, one checks directly that the
corresponding equivariant harmonic map $f$ is bounded near $z=0$.
We omit the details. If $\beta=0$, $\mathfrak X$ and hence $\Psi$ are independent of $\lambda$.
Hence, \eqref{eq:symbobenko} gives a constant map.
\end{proof}

\begin{theorem}\label{thm:twistor-lines}
Let $\alpha$ be generic and let $t>0$ be sufficiently small.
Let $s$ be the real holomorphic section of the parabolic Deligne--Hitchin moduli space
$\mathcal M_{\mathrm{DH}}^{\mathrm{par}}(t\alpha)$ constructed in
Lemma~\ref{lem:realholosec}.
Then $s$ is a twistor line.
\end{theorem}

\begin{proof}
Let $t>0$.
By the previous discussion, the real holomorphic section $s$ of 
$\mathcal M_{\mathrm{DH}}^{\mathrm{par}}(t\alpha)$
given by
$d+t\underline\xi+t\hat\xi(t)$  is a twistor line if and only if
the following two conditions are satisfied:
\begin{itemize}
\item[(1)]
There exists a well--defined gauge transformation
\[
B=B_t\colon \CP^1\setminus\{p_1,\dots,p_n\}\to \Lambda^{\geq0}\mathrm{SL}(2,\C)
\]
such that
\[
(d+t\underline\xi+t\hat\xi(t)). B
=:
d+\omega
\quad \text{with} \quad
\omega\in\Omega^1(\CP^1\setminus\{p_1,\dots,p_n\},\Lambda\mathfrak U).
\]
In this case, $d+\omega$ has the form of an associated family of flat connections
$\nabla_\lambda$ as in \eqref{eq:assocfam}.
\item[(2)]
The equivariant harmonic map obtained via \eqref{eq:symbobenko} is bounded
 in a 
neighbourhood of each $p_j$, $j=1,\dots,n$.
\end{itemize}

For (1), consider a solution of
\[
d\Psi_t+(\,t\underline\xi+t\hat\xi(t)\,)\,\Psi_t=0,
\qquad
\Psi_t(q)=\mathrm{Id},
\]
where $q\in \C\setminus\{p_1,\dots,p_n\}$ is the base point of the fundamental group, as in
Theorem~\ref{thm:real-holomorphic-section-existence}.
On compact subsets $\tilde K$ of the universal covering of
$\CP^1\setminus\{p_1,\dots,p_n\}$, we have the uniform expansion
\[
\Psi_t=\mathrm{Id}+O(t).
\]
Since the Iwasawa decomposition exists on an open subset of
$\Lambda\mathrm{SL}(2,\C)$ containing the identity, there exists $\varepsilon>0$
such that the Iwasawa decomposition of $\Psi_t(z)$ exists for all
$0\leq t<\varepsilon$ and all $z\in \tilde K$.
After possibly choosing 
 $\epsilon$ smaller, we may assume without loss of generality that  $\tilde K$ contains the endpoints of
the lifts of the generators $\gamma_1,\dots,\gamma_n$ of the fundamental group,
all starting at the base point$q$.
Since the monodromy of $\Psi_t$ is contained in $\Lambda\mathcal U$ by construction,
uniqueness of the Iwasawa decomposition implies that the positive factor $B_t$ has trivial monodromy.
In particular, the positive part of the factorization is well defined on the image
$K\subset \CP^1\setminus\{p_1,\dots,p_n\}$ of $\tilde K$ under the covering map.

Next, we analyze the factorization of $\Psi_t$ near a singular point $p_j$.
For $t$ sufficiently small and around $p_j$ fixed, the difference between $\Psi_t$ and the solution for a single puncture, see Lemma \ref{lem:analysisononeleg}, can be uniformly controlled by Gr\"onwall inequality.
(The details are completely analogous to the proof of Lemma~24 in \cite{CHHT} and will
be omitted here.)
Using again that the Iwasawa decomposition exists on an open subset of
$\Lambda\mathrm{SL}(2,\C)$, together with the fact that the Iwasawa decomposition
$\Psi_t=B_tF_t$ exists for solutions with a single puncture by Lemma~\ref{lem:analysisononeleg},
we obtain that, for $t$ sufficiently small, the Iwasawa decomposition exists on the
entire universal covering of $\CP^1\setminus\{p_1,\dots,p_n\}$.
As a consequence, $B=B_t$ is well defined on
$\CP^1\setminus\{p_1,\dots,p_n\}$, and
\[
(d+t\underline\xi+t\hat\xi(t)). B_t
=
dF_t. F_t^{-1}
\in
\Omega^1(\CP^1\setminus\{p_1,\dots,p_n\},\Lambda\mathfrak U).
\]
For point (2), we first note that $F_t$ depends continuously on $t$ and is
(uniformly) close to the unitary part of the factorization \eqref{eq:expIwa} used in Lemma~\ref{lem:analysisononeleg}. Thus, by Lemma~\ref{lem:analysisononeleg} together
with \eqref{eq:symbobenko} implies that the corresponding harmonic map is bounded,
i.e.\ the metric is tame.
This completes the proof.
\end{proof}




\section{The semiclassical limit of the Hitchin metric}\label{sec:scl}
The aim of this section is to describe the semiclassical limit of the Hitchin
metric on moduli spaces of strongly parabolic Higgs bundles as the
parabolic weights $t\alpha$ tend to zero.
More precisely, we show that, after restricting to energy sublevel sets of order
$Ct$ and rescaling the metric by $t^{-1}$, the Hitchin metric converges to the
hyperk\"ahler metric on the corresponding hyperpolygon space.
This comparison is carried out via the degeneration of the associated twistor
lines constructed in Section~\ref{sec:construction}.

\subsection{Energy}

Fix a generic and small weight vector $\alpha$ and a constant $C>0$.
Let $\mathcal M_{Higgs}(\alpha)$ be the moduli space of rank~$2$
stable strongly parabolic Higgs bundles with parabolic weights $\alpha$, equipped
with the Hitchin hyperk\"ahler metric $g^{Hit}$, and let
\[
\mathcal E \colon \mathcal M_{Higgs}(\alpha)\to\R_{\ge 0};\; (\mathcal P,\Phi)\mapsto2i\int_{\CP^1}\tr(\Phi\wedge\Phi^*)
\]
denote the harmonic map energy.

To compute the energy on a $n$-punctured sphere, we first give the contribution of each puncture. Let
\[
A=\sum_{k\geq-1}A_k\lambda^k
=
\begin{pmatrix}
a & b \\
c & -a
\end{pmatrix}
\in\Lambda^{\geq-1}\mathfrak{sl}(2,\C)
\]
 such that $\det(A)=-\alpha_0^2$ for some $\alpha_0>0$.
Write
\[
a=\sum_{k\geq-1} a_k\lambda^k,\qquad
b=\sum_{k\geq-1} b_k\lambda^k,\qquad
c=\sum_{k\geq-1} c_k\lambda^k.
\]
Then
\begin{equation}\label{eq:detcon}
a_{-1}^2+b_{-1}c_{-1}
=
0
=
2 a_0 a_{-1}+b_0 c_{-1}+b_{-1}c_0.
\end{equation}

Motivated by \cite[(110)]{HellerHellerTraizet2025}, see also \cite[Section~5]{BHS},
we define
\[E(A):=
\begin{cases}
-\alpha_0+a_{0}+b_{0}\tfrac{c_{-1}}{a_{-1}} & \text{if } a_{-1}\neq0,\\[0.4em]
-\alpha_0+a_{0}-b_{0}\tfrac{a_{-1}}{b_{-1}} & \text{if } b_{-1}\neq0,\\[0.4em]
-\alpha_0-a_{0}+c_{0}\tfrac{a_{-1}}{c_{-1}} & \text{if } c_{-1}\neq0,\\[0.4em]
0 & \text{if } a_{-1}=b_{-1}=c_{-1}=0.
\end{cases}
\]
By \eqref{eq:detcon}, the different expressions for $E(A)$ coincide in the generic
case $a_{-1}b_{-1}c_{-1}\neq0$. In fact,
if $A_{-1}\neq0$, there exist $0\neq v\in\C^2$ and $0\neq w\in(\C^2)^*$ with $w(v)=0$ such
that $A_{-1}=v\otimes w$.
In this case,
\[
E(A)=-\alpha_0+w(A_0v).
\]

Moreover, for all $g\in\mathrm{SL}(2,\C)$,
\[
E(g^{-1}Ag)=E(A).
\]
Finally, for $\mathfrak A$ as in \eqref{def:mfA} with $x,y$ from \eqref{eq:vwxy} and $\det(\mathfrak A)=-\alpha_0^2$,
we obtain
\begin{equation}\label{eq:Et0}
E(\mathfrak A)=|x|^2+|y|^2,
\end{equation}
which is the standard energy of the Eguchi--Hanson space.

\begin{proposition}\label{pro:energyhitchin}
Let $s$ be a twistor line in $\mathcal M_{\mathrm{DH}}^{\mathrm{par}}(t\alpha)$
 given by a Fuchsian system
$d+t\underline \xi+t\hat\xi(t)$ as in Theorem~\ref{thm:twistor-lines}, with residues
$A_j\in\Lambda^{\geq-1}\mathfrak{sl}(2,\C)$ at its singular points
$p_1,\dots,p_n$.
Then
\[
\mathcal E(s)=\sum_j E(A_j).
\]
\end{proposition}

\begin{proof}
If the weight vector is rational,
$\alpha=(\alpha_1,\dots,\alpha_n)\in\mathbb Q^n$, let $k>0$ be the smallest positive
integer such that $k\alpha\in\Z^n$.
Consider a branched covering $\Sigma_k\to\CP^1$ of degree $k$ which is totally branched
at $p_1,\dots,p_n$.
Then the pullback of $s$ to $\Sigma_k$ gives rise to a smooth solution of Hitchin's
self--duality equations on $\Sigma_k$; see \cite{Nas} or
\cite[Section~3.2.1]{HellerHellerTraizet2025} for details.
The solution upstairs has $k$ times the energy of the solution downstairs.
As in \cite[Section~6.4]{HellerHellerTraizet2025} or \cite[Corollary~18]{HHT1}, the
energy of the smooth solution on $\Sigma_k$ can be computed in terms of residues, and
the statement follows along the same lines.

For small generic weights, we approximate solutions by solutions with rational weights.
Since the solutions depend real analytically on the initial data \cite{KiWi}, the result follows by continuity.
\end{proof}

\begin{remark}
Proposition~\ref{pro:energyhitchin} shows that on twistor lines 
the Hitchin Hamiltonian $\mathcal E$ admits a description in terms of the
residue data.
By \eqref{eq:Et0}, after rescaling by $t^{-1}$, the Hitchin energy reduces in the limit
$t\to0$ to the classical Hamiltonian of rotational $S^1$-action on the hyperpolygon space, and is there a K\"ahler potential for the complex structure $J$.
In this sense, the hyperpolygon Hamiltonian appears as the \emph{semiclassical
Hamiltonian} governing the low--energy regime of the Hitchin metric.
\end{remark}

\subsection{The Hyperk\"ahler metrics}

We now introduce the low--energy subsets and the natural identifications needed
to compare the Hitchin and hyperpolygon hyperk\"ahler metrics.

For $\alpha$ small let $\mathcal X_\alpha$ be the hyperpolygon space with weight $\alpha$, endowed with
its hyperk\"ahler metric $g^{HP}$ and energy function
\[
\mathcal E_{HP}\colon \mathcal X_\alpha\to\R_{\ge 0},
\qquad
\mathcal X_\alpha^{\le C}
:=
\{x\in\mathcal X_\alpha:\mathcal E_{HP}(x)\le C\}.
\]

Let \[
\mathcal M_{Higgs}(t\alpha)_{\le Ct}
:=
\bigl\{(\Pcal,\Phi)\in \mathcal M_{Higgs}(t\alpha)\;:\;
\mathcal E_t(\Pcal,\Phi)\le C\, t\bigr\}
\]
and let $\mathcal H_{Higgs}(t\alpha)$ denote the subset of parabolic Higgs bundles on the
trivial holomorphic bundle. 
By \cite{Meneses26} (see Remark \ref{rem:small-weight} or also \cite[Proposition 7.1]{Hitchin1987}
for the case of compact Riemann surfaces or \cite[Section 6.6]{HellerHellerTraizet2025}
for the case of the 4-punctured sphere) there exist for any small $\alpha$ and  $C>0$ some $\epsilon>0$ such that for all $0<t<\epsilon$
\begin{equation}\label{energy:holomorphic}
\mathcal M_{Higgs}(t\alpha)_{\le Ct}\subset\mathcal H_{Higgs}(t\alpha).\end{equation}
Let \[
\Psi_{\alpha}\colon
\mathcal X_\alpha
\xrightarrow{\;\cong\;}
\mathcal H_{Higgs}(\alpha)
\subset
\mathcal M_{Higgs}(\alpha)
\]
be the natural isomorphism of Godinho--Mandini; see
Theorem~\ref{thm:GM-hyperpolygon-higgs}.
For small $\alpha$, stability does not depend on scaling the weights by a positive real factor $0<t\leq1$.

For $0<t\leq1$, we define the biholomorphic map
\[
\varphi_t\colon
\mathcal H_{Higgs}(\alpha)
\xrightarrow{\;\cong\;}
\mathcal H_{Higgs}(t\alpha),
\qquad
(\mathcal P,\alpha,\Phi)\longmapsto (\mathcal P,t\alpha,t\Phi).
\]
Here $\mathcal P$ denotes the quasi-parabolic structure, i.e.\ the collection of
quasi-parabolic lines up to overall conjugation.
Moreover, for the natural holomorphic symplectic form $\Omega_\alpha$ on
$\mathcal H_{Higgs}(t\alpha)$ we have
\[
\varphi_t^*\Omega_{t\alpha}=t\,\Omega_{\alpha}.
\]

Finally, we set
\[
\Gamma_t
:=
\varphi_t\circ\Psi_\alpha
\colon
\mathcal X_\alpha
\xrightarrow{\;\cong\;}
\mathcal H_{Higgs}(t\alpha).
\]

We can now state and prove our main theorem:
\begin{theorem}\label{thm:semiclassical-hitchin-metric-low-energy-regular}
Let $\alpha$ be generic and small, and let $C>0$ be a regular value of the hyperpolygon energy
$\mathcal E_{HP}$.
Then for all sufficiently small $t>0$ the value $t C$ is a regular value of
$\mathcal E_t$, and the hypersurfaces
\[
\Gamma_t^{-1}\!\bigl(\mathcal E_t^{-1}(t C)\bigr)\subset \mathcal X_\alpha
\]
converge real analytically to $\mathcal E_{HP}^{-1}(C)$ as embedded submanifolds.
In particular, the domains
\[
\Gamma_t^{-1}\bigl(\mathcal M_{Higgs}(t\alpha)_{\le t C}\bigr)
\subset \mathcal X_\alpha
\]
converge to $\mathcal X_\alpha^{\le C}$  as domains with smooth boundary.

Moreover, on every compact subset of $\mathcal X_\alpha{\le C}$ the rescaled metrics
$t^{-1}\Gamma_t^* g^{Hit}_t$ converge real analytically to the hyperpolygon hyperk\"ahler
metric $g^{HP}$. 
\end{theorem}
In particular, the convergence in Theorem \ref{thm:semiclassical-hitchin-metric-low-energy-regular} implies
$C^\infty$ Cheeger--Gromov convergence \cite{ChGr} on  shrinking compact sets.
\begin{proof}
The first claim follows from the existence of twistor lines for $t>0$
(Theorem~\ref{thm:twistor-lines}) together with their limiting behaviour as $t\to0$
(Theorem~\ref{thm:real-holomorphic-section-existence}), and the energy identity
\eqref{eq:Et0} combined with Proposition~\ref{pro:energyhitchin}.
The real--analytic convergence follows from the real--analytic dependence
of the twistor lines and of the twisted-holomorphic symplectic form on the parameter $t$.

The second claim follows similarly, using Theorem~\ref{thm:twistor-hyperpolygon-orbits}
and \eqref{eq:varphilambda} together with
\eqref{eq:Goldman-Kirilov} 
as well as Theorems~\ref{thm:real-holomorphic-section-existence}
and~\ref{thm:twistor-lines}.
\end{proof}

\subsection{Higher order expansion of the Hitchin metric}

Completely analogous to \cite{HellerHellerTraizet2025}, one can compute the Taylor expansion of the
hyperk\"ahler metric coefficients in the deformation parameter $t$ at $t=0.$ Indeed,
expanding the associated family of flat $\lambda$--connections with respect to $t$
yields an iterative system determining the coefficients of the connection $1$--forms.
At each order, these coefficients solve a finite--dimensional linear problem.
The dimension of the linear system for the $k$-th order terms grows linearly in $k$.
Both the linear system and its solutions are expressed explicitly in terms of the
initial hyperpolygon/Higgs data $\underline\xi$ and iterated integrals of  logarithmic
$1$--forms on the $n$--punctured sphere. Equivalently, the Taylor coefficients of the
twisted holomorphic symplectic form---and hence of the hyperk\"ahler metric---can be
written in terms of multiple polylogarithms of increasing depth and weight. The details
are as in \cite[Section~7]{HellerHellerTraizet2025} for
the case of the four--punctured sphere with equal weights.

\begin{theorem}[Higher order expansion of the Hitchin metric]
Let $\alpha$ be small and generic.
For every $C>0$
there exists
$\varepsilon>0$ such that  the Hitchin hyperk\"ahler
metric $g_t$ on 
$\mathcal M_{Higgs}(t\alpha)_{\le t C}$
admits a convergent expansion
\[
g_t \;=\; t\left(g^{HP} \;+\; \sum_{k\geq 1} t^k\, g^{(k)}\right)
\qquad \text{for } 0<t<\varepsilon,
\]
where $g^{HP}$ is the hyperk\"ahler metric on the hyperpolygon space $\mathcal X_\alpha$.

Moreover, for each $k\geq 1$, the coefficient $g^{(k)}$ 
can be written explicitly in terms of iterated integrals of logarithmic
$1$--forms on the $n$--punctured sphere and the coefficients of $\underline\xi\in \mathcal X_\alpha$, i.e. the coefficients of $g^{(k)}$
are given by explicit combinations of multiple polylogarithms of depth and
weight at most $k+1$.
\end{theorem}

In particular, this provides an effective algorithm for computing arbitrarily high
order corrections to the hyperk\"ahler metric in the small--weight regime.


\appendix
\section{Loop algebra lemmas}\label{appendix}
This appendix collects the loop--algebraic linear results needed to solve the
monodromy problem in Section~\ref{sec:construction}.

Let $\mathfrak X\in\Lambda\mathfrak U$ be as in \eqref{eq:mfx}, and define
\begin{equation*}
\begin{split}
\mathcal K_{\mathfrak X}&
:=
\{\,C\in \Lambda\mathfrak{sl}(2,\C)\mid \tr(\mathfrak XC)=0\,\},\;\quad
\mathcal K^{\geq0}_{\mathfrak X}
:=
\mathcal K_{\mathfrak X}\cap\Lambda^{\geq0}\mathfrak{sl}(2,\C),\\
\mathring{\mathcal K}^{\geq0}_X&:=\{ C=\sum_{k\geq0}C^{(k)}\lambda^k\in\mathcal K^{\geq0}_{\mathfrak X}\mid
C^{(0)}\begin{pmatrix}1\\0\end{pmatrix}\subset \C \begin{pmatrix}1\\0\end{pmatrix}\}.
\end{split}
\end{equation*}
$\mathcal K_X$, $\mathcal K^{\geq0}_{\mathfrak X}$ and $\mathring{\mathcal K}^{\geq0}_X$ are closed subspaces of $\Lambda\mathfrak{sl}(2,\C)$.

\begin{proposition}\label{prop:solutionatpj}
Let  $\pi^\mathfrak H$ be the orthogonal projection onto $\mathfrak H\subset \Lambda\mathfrak{sl}(2,\C)$.
Then,
\[\pi^\mathfrak H\colon \mathcal K^{\geq0}_{\mathfrak X}\to \Lambda\mathfrak H\cap \mathcal K_{\mathfrak X}\]
is an  isomorphism between real Banach spaces.
\end{proposition}
\begin{proof}
Since the projection $\pi^\mathfrak H$ is continuous, we only need to show bijectivity.
Note that
\[
\Lambda\mathfrak U\cap \Lambda^{\geq0}\mathfrak{sl}(2,\C)=\mathfrak{su}(2)
\]
Thus, for $\beta\neq0$, injectivity follows from
$
\mathfrak{su}(2)\cap \mathcal K_{\mathfrak X}=\{0\}.
$
If $\beta=0$, 
\[\mathfrak{su}(2)\cap \mathcal K_{\mathfrak X}=\{\begin{pmatrix}0&\gamma\\-\bar\gamma& 0\end{pmatrix}\mid \gamma\in\C\}\]
and therefore
$\Lambda\mathfrak U\cap \mathring{\mathcal K}_{\mathfrak X}=\{0\}$ implies injectivity.

The proof of surjectivity is given in 
Lemma~\ref{lem:existPp} 
 below.
\end{proof}

We first establish some
elementary lemmas.

\begin{lemma}\label{lem:YhatY}
For every $Y\in\mathcal K_{\mathfrak X}$ there exists
$\hat Y\in \Lambda\mathfrak{sl}(2,\C)$ such that
\[
Y=[\mathfrak X,\hat Y].
\]
Moreover, $\hat Y$ is unique up to adding $f\,\mathfrak X$, where
$f\colon S^1\to\C$ is real analytic.
The map $Y\mapsto \hat Y$ can be chosen to be linear and continuous.
\end{lemma}

\begin{proof}
Uniqueness up to adding $f\,\mathfrak X$ is immediate, since $\mathfrak X$ is
non--vanishing on $S^1$, and the corresponding statement holds in
$\mathfrak{sl}(2,\C)$.
For existence, consider the operator
\begin{equation}\label{def:Dcim}
\mathcal D\colon
\Lambda\mathfrak{sl}(2,\C)\to \Lambda\mathfrak{sl}(2,\C),
\qquad
Y\longmapsto D^{-1}YD,
\end{equation}
where $D=\mathrm{diag}(1,\lambda)$.
Clearly, $\mathcal D(YZ)=(\mathcal D Y)(\mathcal D Z)$ and
$\mathcal D[Y,Z]=[\mathcal D Y,\mathcal D Z]$.
Furthermore, $\mathcal D\mathfrak X\in\mathfrak{sl}(2,\C)$ is independent of $\lambda$ by construction.
Writing $\mathcal D Y=\sum_k c_k\lambda^k$, we 
solve for each coefficient $c_k=[\mathcal D\mathfrak X, \hat c_k]$, to obtain the result.
\end{proof}

\begin{lemma}\label{lem.red}
Let $\mathfrak X\in\Lambda\mathfrak U$ be as above, and let
$Y=[\mathfrak X,\hat Y]\in\mathcal K_\mathfrak X$.
Then, the respective unitary and hermitian parts satisfy
\[
Y^u=[\mathfrak X,\hat Y^u]
\quad\text{and}\quad
Y^h=[\mathfrak X,\hat Y^h].
\]
\end{lemma}

\begin{proof}
This is a direct consequence of \eqref{eq:real_invol} and $\mathfrak X^*=\mathfrak X$.
\end{proof}

\begin{lemma}\label{lem:existPp}
Let $\mathfrak X\in\Lambda\mathfrak U$ be as above, and let
$H
\in\mathcal K_\mathfrak X\cap\Lambda\mathfrak H$. 
Then there exists
\[
 P\in \mathring{\mathcal K}_\mathfrak X^{\geq0}
\]
such that
\[\pi^\mathfrak HP=H.\]
\end{lemma}

\begin{proof}
By Lemma \ref{lem:YhatY} and Lemma \ref{lem.red} we can 
write $H=[X,\hat H]$ for some $\hat H\in\Lambda\mathfrak H.$
Decompose
$\hat H=\hat H^-+\hat H^0+\hat H^+$.

Then $(\hat H^0)^*=-\hat H^0$, i.e.\ $\hat H^0\in i\,\mathfrak{su}(2)$, and
$\hat H^+=-(\hat H^-)^*$.
Define $\hat P^+:=\hat 2H^+$.

Then 
\[
[\mathfrak X,\hat P^+]\in\Lambda^{\geq0}\mathfrak{sl}(2,\C)
\]
and
\[\pi^\mathfrak H[\mathfrak X,\hat P^+]=[\mathfrak X,\pi^\mathfrak H\hat P^+]=[\mathfrak X,\hat H^+- (\hat H^+)^*]=[\mathfrak X,\hat H^+ +\hat H^-].\]

Let
\[
\hat H^0=
\begin{pmatrix}
a & b \\
\bar b & -a
\end{pmatrix},
\]
which is Hermitian symmetric by assumption, so in particular $a\in\R$.
Define
\[
\hat P^0
:=
\begin{pmatrix}
0 & 2b \\
-2\frac{a\bar\beta}{\mu}\lambda & 0
\end{pmatrix}.
\]
Then (by computation)
\[
[\mathfrak X,\hat P^0]\in \Lambda^{\geq0}\mathfrak{sl}(2,\C)
\]
and
\[\pi^\mathfrak H [\mathfrak X,\hat P^0]=[\mathfrak X,\hat H^0]\]
since $a\in\R$ and $\mu\in i\,\R$.
Thus, using Lemma  \ref{lem.red} and \eqref{eq:real_invol},
\[
P:=[\mathfrak X,\hat P]\quad \text{for}\quad \hat P:=\hat P^0+\hat P^+
\]
satisfies the stated equations, and that $P\in\mathring{\mathcal K}^{\geq0}_{\mathfrak X}$
can be directly verified.
\end{proof}

\section*{Acknowledgements}{\small
The authors acknowledge the use of ChatGPT as an editorial tool for language editing,
rephrasing, and improving the clarity and structure of the exposition.
The second author acknowledges financial support from the Beijing Natural Science Foundation
IS23003. The third author acknowledges financial support from the SPP 2026 "Geometry at infinity" 
DFG-priority programme, and would like to warmly thank Carlos Florentino for fruitful
discussions on moduli of hyperpolygons, as well as BIMSA for its hospitality during the development of this work.}



\begin{thebibliography}{99}

\bibitem{AW98}
S.~Agnihotri and C.~Woodward, {\em Eigenvalues of products of unitary matrices and quantum
Schubert calculus}, Math. Res. Lett. \textbf{5} (1998) 817--836.

\bibitem{AleMale}
A. Alekseev, A. Malkin, {\em Symplectic structure of the moduli space of flat connections on a Riemann
surface}, Commun. Math. Phys. \textbf{169} (1995) 99--119.

\bibitem{AlGo} {\sc D. Alfaya, T. L. G\'omez}, {\em Torelli theorem for the parabolic Deligne-Hitchin moduli space}, J. Geom. Phys. \textbf{123} (2018) 448--462.

\bibitem{Biquard98}
O.~Biquard, \emph{Twisteurs des orbites coadjointes et m\'etriques hyper-pseudok\"ahl\'eriennes},
 Bull. Soc. Math. Fr. \textbf{126} (1), 79--105 (1998).

\bibitem{BiquardMinerbe}
O.~Biquard and V.~Minerbe,
\emph{A Kummer construction for gravitational instantons},
Commun. Math. Phys. \textbf{308} no.~3 (2011) 773--794.

\bibitem{BBDH} I. Biswas, S. Bradlow, S. Dumitrescu, S. Heller, {\em Uniformization of branched
surfaces and Higgs bundles}, Int. J. Math. \textbf{32} no.~13 (2021).

\bibitem{BiFloGoMa} I. Biswas, C. Florentino,  L. Godinho, A. Mandini, {\em Symplectic form on hyperpolygon spaces}, Geom. Dedicata \textbf{179} no.~1 (2015) 187--195.

\bibitem{Boalch08} Ph. Boalch, {\em Irregular connections and Kac--Moody root systems,} 2008, arXiv:0806.1050.

\bibitem{Boalch12} Ph. Boalch, {\em  Hyperk\"ahler manifolds and nonabelian Hodge theory on (irregular) curves,} 2012,
arXiv:1203.6607.

\bibitem{Boalch} Ph. Boalch, {\em Geometry and braiding of Stokes data; fission and wild character varieties}, 
Ann. Math. \textbf{179} no.~1 (2014) 301--365.

\bibitem{Boalch18} Ph. Boalch, {\em Wild character varieties, meromorphic Hitchin systems and Dynkin diagrams,}
In Geometry and Physics Vol. II, 433--454, Oxford University Press (2018).

\bibitem{BHS} A.I. Bobenko, S. Heller, N. Schmitt,
{\em Minimal reflection surfaces in $S^3$.
Curvature line foliations and
new examples based on fundamental pentagons,}
J. Geom.  Phys. \textbf{209} (2025) 105407.

\bibitem{BY96} H.~Boden, K.~Yokogawa, {\em Moduli spaces of parabolic Higgs bundles and parabolic K(D) pairs
over smooth curves I,} Int. J. Math. \textbf{7} no.~05 (1996) 573--598.

\bibitem{BY99} H.~Boden, K.~Yokogawa, {\em Rationality of moduli spaces of parabolic bundles}, J. London Math.
Soc. \textbf{59} no.~2 (1999) 461--478.

\bibitem{Calabi}E.~Calabi, {\em M\'etriques k\"ahl\'eriennes et fibr\'es holomorphes,} Ann. Sci. \'Ec. Norm. Sup\'er. \textbf{12} no.~2,
(1979) 269--294.

\bibitem{CHHT}
S.~Charlton, L.~Heller, S.~Heller, and M.~Traizet,\
\emph{Minimal surfaces and alternating multiple zetas},\
arXiv:2407.07130.

\bibitem{ChGr}
J. Cheeger, M. Gromov, {\em 
Collapsing Riemannian manifolds while keeping their curvature bounded I}, J. Differ.
Geom. \textbf{23} (1986) 309--346.

\bibitem{Donald} S. Donaldson, {\em Twisted harmonic maps and the self-duality equations}, Proc. London Math. Soc. (3) \textbf{55} no.~1 (1987) 127--131.


\bibitem{DPW} J. Dorfmeister, F. Pedit, Wu,
\emph{Weierstrass type representation of
  harmonic maps into symmetric spaces}, Comm. Anal. Geom. \textbf{6} no.~4 (1998) 633--668.

\bibitem{FR16} J.~Fisher and S.~Rayan, {\em Hyperpolygons and Hitchin systems}, Int. Math. Res. Not. Vol. 2016,
no. 6 (2016) 1839--1870.


\bibitem{Fred} 
L.~Fredrickson,
\emph{Exponential decay for the asymptotic geometry of the Hitchin metric},
Commun. Math. Phys. \textbf{375} (2020), no.~2, 1393--1426.

\bibitem{FMSW}
L. Fredrickson, R. Mazzeo, J. Swoboda, H. Weiss, {\em  Asymptotic geometry of the moduli space of
parabolic $\mathrm{SL}(2,\C)$-Higgs bundles}, J. London Math. Soc. 1--72 (2022).

\bibitem{FredYae26} 
L.~Fredrickson, A.~Yae, Private email communication, 2025-12-22.


\bibitem{GMN}
D.~Gaiotto, G.~W. Moore, and A.~Neitzke,
 {\em Wall-crossing, Hitchin systems, and the WKB approximation},
Adv. Math. \textbf{234} (2013), 239--403.

\bibitem{GoMa} L. Godinho, A. Mandini, {\em Hyperpolygon spaces and moduli spaces of parabolic Higgs bundles}, Adv. Math. \textbf{244} (2013) 465--532.




\bibitem{HHS}
L. Heller, S. Heller, N. Schmitt, {\em Navigating the Space of Symmetric CMC Surfaces,}
J. Differ.
Geom., {\bf 110}, no. 3 (2018), pp. 413--455.

    \bibitem{HH}
L. Heller, S. Heller, \emph{Higher solutions of Hitchin's self-duality equations}, J. Int. Sys. \textbf{5} (2020).

\bibitem{HHT1} L.~Heller, S.~Heller, and M.~Traizet,\
\emph{Area estimates for High genus Lawson surfaces via DPW}, J. Differ.
Geom. \textbf{124} (2023) 1--35.

\bibitem{HellerHellerTraizet2025}
L.~Heller, S.~Heller, and M.~Traizet,\
\emph{Loop group methods for the non-abelian Hodge correspondence on a 4-punctured sphere},\
Math. Ann. \textbf{392} no.~3 (2025) 2981--3064. 

\bibitem{Heller2025}
S.~Heller,
\emph{Self-duality solutions near infinity},
unpublished manuscript, 2025.

\bibitem{Hitchin1987}
N.~J.~Hitchin,
\emph{The self-duality equations on a Riemann surface},
Proc. London Math. Soc. (3) \textbf{55} no.~1 (1987) 59--126.

\bibitem{Hitchin91}
N.~J.~Hitchin, {\em Hyperk\"ahler manifolds}, S\'eminaire N. Bourbaki, 1991-1992, exp. no 748, 137--166.

\bibitem{HKLR}  N. J. Hitchin, A. Karlhede, U. Lindstr\"om, and M. Rocek, {\em Hyperk\"ahler 
Metrics and Supersymmetry}, Commun. Math. Phys. {\bf 108} (1987) 535--589.

\bibitem{KiWi} S. Kim, G. Wilkin, {\em Analytic convergence of harmonic metrics for parabolic Higgs bundles},
J. Geom.  Phys. \textbf{127} 55--67 (2018).

\bibitem{Kirwan}
F.~C.~Kirwan,
\emph{Cohomology of Quotients in Symplectic and Algebraic Geometry},
Mathematical Notes, vol.~31,
Princeton University Press, Princeton, NJ, 1984.

\bibitem{Klyachko}
A~ Klyachko, \emph{Spatial polygons and stable configurations of points in the projective line}, In Algebraic geometry and its applications (Yaroslavl, 1992),
Aspects Math., Vieweg, Braunschweig (1994) 67--84.

\bibitem{Konno} H. Konno, {\em Construction of the moduli space of stable parabolic Higgs bundles on a Riemann surface},
J. Math. Soc. Japan, \textbf{45} No. 2 (1993) 253--276.

\bibitem{Konno2} H. Konno, {\em On the cohomology ring of the hyperK\"ahler analogue of the polygon spaces}, 
Contemp. Math. \textbf{309} Amer. Math. Soc,  RI (2002) 129--149.

\bibitem{MSWW}
R. Mazzeo, J. Swoboda, H. Wei\ss , F. Witt, {\em Asymptotic geometry of the Hitchin metric},
Comm. Math. Phys. \textbf{367} no. 1 (2019) 151--191.

\bibitem{MeS} V.B. Mehta, C.S. Seshadri, {\em Moduli of vector bundles on curves with parabolic structures}, Math. Ann.
\textbf{248} (1980) 205--239.

\bibitem{MenesesSIGMA}
C.~Meneses,
\emph{Geometric models and variation of weights on moduli of parabolic Higgs bundles over the Riemann sphere: a case study},
SIGMA 
\textbf{18} (2022) 062, 41 p.

\bibitem{Meneses26}
C.~Meneses, \emph{Stability variations for moduli of rank 2 parabolic bundles on the Riemann sphere I. Geometric models for wall-crossing}, in preparation.

\bibitem{Mochizuki2006}
T.
Mochizuki, {\em Kobayashi-Hitchin correspondence for tame harmonic bundles and an application},
Ast{\'e}risque,
{\bf 309},
(2006).

\bibitem{Mo} T.
Mochizuki, {\em 
Asymptotic behaviour of tame harmonic bundles and an application to pure twistor D-modules. I.}, 
Mem. Am. Math. Soc. \textbf{869} (2007) 324 p.

\bibitem{Nak} H. Nakajima, {\em Quiver varieties and Kac-Moody algebras}, Duke Math. J. \textbf{91} (1998) 515--560.

\bibitem{Nas} B. Nasatyr,  B. Steer, {\em Orbifold Riemann surfaces and the Yang-Mills-Higgs equations}, Ann. Scuola Norm. Sup. Pisa \textbf{22} no.~4 (1995) 595--643.

\bibitem{OSWW} A. Ott, J. Swoboda, R. Wentworth, M. Wolf, {\em Higgs bundles, harmonic maps and pleated surfaces}, Geom.
Topol. \textbf{28} no. 7 (2024) 3135--3220.

\bibitem{PressleySegal}
A.~Pressley and G.~Segal,
\emph{Loop Groups},
Oxford Mathematical Monographs,
Oxford University Press, Oxford, 1986.

\bibitem{RS21}S.~Rayan and L.~P.~Schaposnik, {\em Moduli spaces of generalized hyperpolygons}, Q. J. Math. \textbf{72}
no. 1-2 (2021) 137--161.

\bibitem{Sim1} C. Simpson, {\em Constructing variations of Hodge structure using Yang-Mills theory and applications to
uniformization},  J. Amer. Math. Soc. \textbf{1} (1988) 867--918.
\bibitem{Sim2} C. Simpson, {\em Harmonic bundles on noncompact curves}, J. Amer. Math. Soc. \textbf{3} no.~3 (1990) 713--770.
\bibitem{Sim3} C. Simpson, {\em  The Hodge filtration on nonabelian cohomology}. Algebraic geometry--Santa Cruz 1995, 217--281, Proc. Sympos. Pure Math. 62, Amer. Math. Soc., Providence, RI, (1997).
\bibitem{Sim21} C. Simpson, {\em The twistor geometry of parabolic structures in rank two}, Proc. Indian Acad. Sci., Math. Sci. \textbf{132} No. 2, Paper No. 54, 26 p. (2022).
\bibitem{Thaddeus}
M. Thaddeus, {\em Variation of moduli of parabolic Higgs bundles}, J. Reine Angew. Math. \textbf{547} (2002) 1--14.
\end{thebibliography}
\end{document}